\documentclass[10pt]{article}
\usepackage{amssymb}
\usepackage{amsmath}
\usepackage[latin1]{inputenc}
\usepackage{graphicx,color}
\usepackage{tikz}
\usepackage{color}
\newtheorem{theorem}{Theorem}[section]

\newtheorem{lemma}{Lemma}[section]
\newtheorem{proposition}{Proposition}[section]

\newtheorem{corollary}{Corollary}[section]
\newcommand{\fim}{\hfill\rule{2mm}{2mm}}

\begin{document}
\title{
\vspace{0.5in}  {  Multiplicity of negative-energy solutions for singular-superlinear Schr\"odinger  equations with indefinite-sign potential}}

\author{
{\large Carlos Alberto Santos$^a$}\footnote{
the support of CAPES/Brazil Proc.  $N^o$ $2788/2015-02$.}\,\, ~~~~~~ {\large Ricardo Lima Alves$^b$}  ~~~~~~ {\large Kaye Silva $^{c}$}
\hspace{2mm}\\
{\it\small $^a$Universidade de Bras\'ilia, Departamento de Matem\'atica}\\
{\it\small   70910-900, Bras\'ilia - DF - Brazil}\\
{\it\small $^b$ Universidade de Bras\'ilia, Departamento de Matem\'atica}\\
{\it\small   70910-900, Bras\'ilia - DF - Brazil}\\
{\it\small $^{c}$ Instituto de Matemática e Estatística, Universidade Federal de Goiás}\\
{\it\small   74001-970, Goi\^ania - Go - Brazil}\\
{\it\small e-mails: csantos@unb.br, ricardoalveslima8@gmail.com, kayeoliveira@hotmail.com }\vspace{1mm}\\
}

\date{}
\maketitle \vspace{-0.2cm}

\begin{abstract}
	We are concerned with the multiplicity of positive solutions  for the singular superlinear and subcritical Schr\"odinger equation 
	$$
	\begin{array}{c}
	-\Delta u +V(x)u=\lambda a(x)u^{-\gamma}+b(x)u^{p}~\mbox{in}~ \mathbb{R}^{N},
	\end{array}
	$$
	beyond the Nehari extremal value, as defined in Il'yasov \cite{I}, when the potential  $b \in L^{\infty}(\mathbb{R}^{N})$ may change its sign, $0<a\in L^{\frac{2}{1+\gamma}}(\mathbb{R}^{N})$, $V$ is a positive continuous function, $N\geq 3$ and $\lambda>0$ is a real parameter. The main difficulties come from the non-differentiability of the energy functional and the fact that the intersection of the boundaries of the connected components of the Nehari set is non empty.  We overcome these difficulties by exploring topological structures of that boundary to build non-empty sets whose boundaries have empty intersection and minimizing over them by controlling the energy level.
\end{abstract}


%
%
%
\section{\bf Introduction}
\def\theequation{1.\arabic{equation}}\makeatother
\setcounter{equation}{0}

In this paper, we deal with results of multiplicity and non-existence of $H^{1}(\mathbb{R}^{N})$-solutions for the problem
\begin{equation}\label{pq}\tag{$P_{\lambda}$}
	\left\{
	\begin{aligned}
		-\Delta u +V(x)u = \lambda a(x)u^{-\gamma}+ b(x)u^{p}
		~\mbox{in}~ \mathbb{R}^{N}, \\
		u> 0,~\mathbb{R}^{N},~\int_{\mathbb{R}^{N}}Vu^2< \infty,~ u\in H^{1}(\mathbb{R}^{N}), \nonumber
	\end{aligned}
	\right.
\end{equation}
where 
$ 0<a\in L^{\frac{2}{1+\gamma}}(\mathbb{R}^{N}),~b^{+} \neq 0, b\in L^{\infty}(\mathbb{R}^{N})$, $V:\mathbb{R}^{N}\to \mathbb{R}$ is a positive continuous function, $0< \gamma<1<p<2^{\ast}-1$, $N\geq 3$  and
$\lambda>0$ is a real positive parameter. 

Since the pioneering work by Fulks-Maybee \cite{FM} on singular problems, this kind of subject has drawn the attention of several researchers. They showed that if $\Omega \subset \mathbb{R}^{3}$ is a bounded region of the space occupied by an
electrical conductor, then $u$ satisfies the equation
$$cu_{t}-k\Delta u=\frac{E^{2}(x,t)}{t^{\gamma}},$$
where $u(x,t)$ denotes the temperature at the point $x\in \Omega$ and time $t$, $E(x,t)$ describes the local voltage drop, $t^{\gamma}$ with $\gamma>0$ is the electrical resistivity, $c$ and $k$ are the specific heat and the thermal conductivity of the conductor, respectively. 

Due to the applications or mathematical purposes, the issues of multiplicity (both local and global) of solutions for elliptic problems have been largely considered in the last decades. In 1994, Ambrosetti-Brezis-Cerami in \cite{ABC}, by exploring the sub and super solution method and Mountain Pass Theorem, proved a global multiplicity result, i.e.,  there exists a $\Lambda>0$ such that the problem
\begin{equation}\label{js}\tag{$Q_{\lambda}$}
	\left\{
	\begin{aligned}
		&-\Delta u  = \lambda a(x)\vert u \vert^{\gamma-2} u+ b(x)\vert u \vert^{p-2} u
		~\mbox{in}~ \Omega, \\
		&u> 0~\mbox{in }\Omega,~u=0~\mbox{on }\partial\Omega, \nonumber
	\end{aligned}
	\right.
\end{equation}
admits at least two positive solutions for $0<\lambda <\Lambda$,  at least one solution for $\lambda=\Lambda$ and   no  solution at all for $\lambda>\Lambda$, when  $\Omega \subset \mathbb{R}^N$ is a smooth bounded domain, $a=b=1$, $ 1<\gamma<2<p<2^{\ast}$ and $2^{\ast}$ is  the critical Sobolev exponent. Considering more general operators and hypothesis, problem \eqref{js} was further generalized in Figueiredo-Gossez-Ubilla \cite{FGU,FGU1}.

Recently, some authors have studied  problems like ($Q_{\lambda}$)  by using only variational methods, to wit, the Nehari manifold and the fibering method of Pohozaev \cite{P} (see \cite{HSS, SM, SL1, SWL}). In 2018, Silva-Macedo in \cite{SM} took advantage of the $C^1$-regularity of the energy functional associated to problem ($Q_{\lambda}$)  with $a=1$ to refine the Nehari's classical  arguments and to show multiplicity of solutions beyond the Nehari's extremal value 
\begin{equation}\label{E3}
	\lambda_{\ast}=\left(\frac{2-\gamma}{p-\gamma}\right)^{\frac{2-\gamma}{p-2}}\left(\frac{p-2}{p-\gamma}\right)\displaystyle \inf_{0\lneqq u \in H_0^1(\Omega),\int_{\Omega}b|u|^{p+1}>0}\frac{\left(||u||^{2}\right)^{\frac{p-\gamma}{p-2}}}{\left[\int_{\Omega}b|u|^{p}\right]^{\frac{2-\gamma}{p-2}}\left[\int_{\Omega}a|u|^{\gamma}\right]},
\end{equation}
as defined in Il'yasov \cite{I}. 

Similar issues have been considered for singular problems of the type
\begin{equation}\label{ks}\tag{$R_{\lambda}$}
	\left\{
	\begin{aligned}
		&-\Delta u  = \lambda a(x)u^{-\gamma} + b(x)u^{p} 
		~\mbox{in}~ \Omega, \\
		&u> 0~\mbox{in }\Omega,~u=0~\mbox{on }\partial\Omega, \nonumber
	\end{aligned}
	\right.
\end{equation}
where $0<\gamma<1<p<2^{\ast}-1, \Omega \subset \mathbb{R}^{N}$ is a smooth bounded domain. In 2003, Haitao in \cite{H} proved a global multiplicity result for Problem ($R_{\lambda}$) with $a=b=1$ by combining sub-supersolution and variational methods. In 2008, Yijing-Shujie in \cite{SL1} considered the problem $(R_{\lambda})$ with potentials $ a,b \in C(\overline{\Omega})$ satisfying $a\geq 0, a\not\equiv 0 $ and $b$ may change sign. They proved a local multiplicity result, i.e., there exists a  $\Lambda>0$ such that the problem $(R_{\lambda})$ admits at least two non-negative solutions for each $\lambda \in (0,\Lambda)$. Still in this context of bounded smooth domains,  we refer the reader to \cite{PRR, CRT, LM, SWL} where different techniques, more general operators and non-linearities are considered.

On $\mathbb{R}^N$ there are a few results related with existence, multiplicity and non-existence of solutions for Problems like ($Q_{\lambda}$). By using the sub and super solution method combined with perturbation arguments, the authors Carl-Perera \cite{CP}, Gonçalves-Santos \cite{GMS}, C\^irstea-R\v{a}dulesco \cite{CR}, Edelson \cite{E} proved existence of $C^1(\mathbb{R}^N)$-solutions. 

With respect to the variational techniques point of view, as far as we know, there is just one, to wit,  Liu-Guo-Liu \cite{LL} in 2009 proved a local multiplicity result of $D^{1,2}(\mathbb{R}^N)$-solutions for the equation
$$-\Delta u=a(x)u^{-\gamma}+\lambda b(x)u^{p},~x\in \mathbb{R}^{N},u>0, $$
where $N\geq 3,  \lambda>0, 0<\gamma<1<p<2^{\ast}-1$ and $b$ may change sign. They combined a local minimization over the ball with an extension of the Mountain Pass Theorem for nonsmooth functionals (see Canino-Degiovani \cite{CD}). Due to the their techniques, it is not hard to see that their extremal value that still guarantees multiplicity of solutions is less than
\begin{equation}
	\label{eq41}
	\hat{\lambda}=(1-\gamma)\frac{(p+1)^{\frac{1+\gamma}{p-1}}}{2^{\frac{p+\gamma}{p-1}}}\left(\frac{1+\gamma}{p+\gamma}\right)^{\frac{1+\gamma}{p-1}}\left(\frac{p-1}{p+\gamma}\right)\displaystyle \inf_{0\lneqq u \in H_0^1(\mathbb{R}^N),\int_{\mathbb{R}^N}b|u|^{p+1}>0}\frac{\left(||u||^{2}\right)^{\frac{p+\gamma}{p-1}}}{\left[\int_{\mathbb{R}^N}b|u|^{p+1}\right]^{\frac{1+\gamma}{p-1}}\left[\int_{\mathbb{R}^N}a|u|^{1-\gamma}\right]},
\end{equation}
because they were able to show multiplicity of solutions just in the $\lambda$-variation of the parameter $\lambda$ that still produces the second solution with positive energy. 

By using a new approach, we were able to prove multiplicity of solutions for Problem ($P_{\lambda}$) beyond $\hat{\lambda}$, that necessarily implies that all the solutions found by this method have negative energies. Besides this, we were also able to characterize a $\lambda$-behavior of the energy functional along the solutions. 

To state our main results, let us assume that $V:\mathbb{R}^{N}\to \mathbb{R}$ is a positive continuous function that satisfies
\begin{itemize}
	\item[$(V)_{0}$] $V_{0}:=\displaystyle \inf_{x\in \mathbb{R}^{N}}V(x)>0$,
	and  one of the following conditions:
	\begin{itemize}
		\item[$(i)$] $\displaystyle \lim_{|x|\to \infty}V(x)=\infty $;
		\item[$(ii)$] ${1}/{V}\in L^{1}(\mathbb{R}^{N})$;
		\item[$(iii)$] for each $ M>0$ given the  $|\left\{x\in \mathbb{R}^{N}:V(x)\leq M\right\}|<\infty .$
	\end{itemize}
\end{itemize}

Define 
\begin{equation}
	\label{13}
	X=\left\{u\in H^{1}(\mathbb{R}^{N}): \int_{\mathbb{R}^{N}}V(x)u^{2}dx<\infty\right\},
\end{equation}
and observe that $\Phi_{\lambda}:X\rightarrow \mathbb{R}$ defined by
\begin{equation}
	\label{14}
	\Phi_{\lambda}(u)=\frac{1}{2}\int_{\mathbb{R}^{N}}(|\nabla u|^{2}+V(x)u^{2})dx-\frac{\lambda}{1-\gamma}\int_{\mathbb{R}^{N}}a(x)|u|^{1-\gamma}dx-\frac{1}{p+1}\int_{\mathbb{R}^{N}}b(x)|u|^{p+1}dx,
\end{equation}
is well-defined and continuous. One of the main difficulties of this work is the lack of G\^ateaux differentiability of the energy functional $\Phi_\lambda$, which is due to the presence of the singular term.

We say that $u\in X$ is a solution of \eqref{pq} if
\begin{equation*}
\int_{\mathbb{R}^{N}}\nabla u \nabla \psi+V(x)u \psi dx=\lambda \int_{\mathbb{R}^{N}}a(x)u^{-\gamma}\psi dx+\int_{\mathbb{R}^{N}}b(x)u^{p}\psi dx~\mbox{for all }\psi \in X.
\end{equation*}

Related to the structure of the functional $\Phi_{\lambda}$, let us set (see Hirano-Saccon-Shioji \cite{HSS} and Il'yasov \cite{I})
\begin{equation}\label{E3}
	\lambda_{\ast}=\left(\frac{1+\gamma}{p+\gamma}\right)^{\frac{1+\gamma}{p-1}}\left(\frac{p-1}{p+\gamma}\right)\displaystyle \inf_{0\lneqq u \in H_0^1(\mathbb{R}^N),\int_{\mathbb{R}^N}b|u|^{p+1}>0}\frac{\left(||u||^{2}\right)^{\frac{p+\gamma}{p-1}}}{\left[\int_{\mathbb{R}^N}b|u|^{p+1}\right]^{\frac{1+\gamma}{p-1}}\left[\int_{\mathbb{R}^N}a|u|^{1-\gamma}\right]},
\end{equation}
which relates with $\hat{\lambda}>0$ defined at  (\ref{eq41}) by
$$\hat{\lambda}=(1-\gamma)\frac{(p+1)^{\frac{1+\gamma}{p-1}}}{2^{\frac{p+\gamma}{p-1}}}\lambda_{\ast}<\lambda_{\ast}.$$

Our first result is
\begin{theorem}\label{TP1} 
	Suppose that  $0<\gamma<1<p<2^{\ast}-1,0< a\in L^{\frac{2}{1+\gamma}}(\mathbb{R}^{N}),~ b\in L^{\infty}(\mathbb{R}^{N}),~b^{+} \neq 0,~ (V)_{0}$ and $\left[{a}/{b}\right]^{\frac{1}{p+\gamma}}\notin X$ if $b>0$ in $\mathbb{R}^{N}$ hold.  Then there exists an $\epsilon>0$ such that the problem $(P_{\lambda})$ has at least two positive solutions $w_{\lambda}, u_{\lambda} \in X$ for each $0<\lambda<\lambda_{\ast}+\epsilon$ given. Besides this, we have:
	\begin{itemize}
		\item[$a)$] $ \frac{d^2\Phi_\lambda}{dt^2}(t u_\lambda){\big\vert_{t=1}}>0$ and $ \frac{d^2\Phi_\lambda}{dt^2}(t w_\lambda){\big\vert_{t=1}}<0$ for all $0<\lambda<\lambda_\ast + \epsilon$,
		\item[$b)$]   there exists a constant $c>0$ such that $||w_{\lambda}||\geq c$ for all $0<\lambda < \lambda_{\ast}+\epsilon$,
		\item[$c)$] $u_\lambda$ is a ground state solution for all $0<\lambda \leq \lambda_{\ast}$, $\Phi_\lambda(u_\lambda) <0$ for all $0<\lambda<\lambda_\ast + \epsilon$ and $\displaystyle \lim_{\lambda \to 0}||u_{\lambda}||=0$,
		\item[$d)$] the applications $\lambda \longmapsto \Phi_\lambda(u_\lambda)$ and $\lambda \longmapsto \Phi_\lambda(w_\lambda)$ are decreasing for $0<\lambda < \lambda_{\ast}+\epsilon$ and are left-continuous ones for $0<\lambda < \lambda_{\ast}$,
		\item[$e)$]  $\Phi_{\lambda} (w_\lambda) >0$ for  $0<\lambda<\hat{\lambda}$, $\Phi_{\hat{\lambda}}(w_{\hat{\lambda}}) =0$  and $\Phi_\lambda(w_\lambda) <0$ for  $\hat{\lambda}<\lambda<\lambda_\ast + \epsilon$ $(
		$see $\hat{\lambda}$ in $(\ref{eq41}))$,
	\end{itemize}
\end{theorem}

The second result gives us an estimate on how big the number $\epsilon>0$ can be, under additional assumptions on $a$ and $b$. 

\begin{theorem}\label{TP2}
	Suppose that the hypotheses of Theorem \ref{TP1} hold. Moreover, assume that there exists a smooth bounded open set  $\Omega \subset \mathbb{R}^{N}$ such that $b>0$ in $\Omega$ and $a\in L^{\infty}(\Omega)$. Then there exists $\lambda^\ast>0$ such that the problem $(P_{\lambda})$ has no solution at all for $\lambda >\lambda^\ast$. Moreover, we have the exact estimate
	$$0<\lambda_{\ast}< \lambda^\ast=\lambda_{1}^{\frac{p+\gamma}{p-1}} \left(\frac{\gamma+1}{p-1}\right)^{\frac{\gamma+1}{p-1}}\left(\frac{p-1}{p+\gamma}\right)^{\frac{p+\gamma}{p-1}},$$
	where $\lambda_{1}>0$ is given in Lemma $\ref{APB}.$
\end{theorem}
Some comments  are in order now: 
\begin{itemize}
	\item[a)] Theorem \ref{TP1} is new in the literature by showing multiplicity of solutions with negative energies as well,
	\item[b)] traditionally two solutions are found by minimizing the energy functional over connected components of the Nehari manifold which are separated in the sense that their boundaries have disjoint intersection. In this work we go further, because we find solutions  in the case where such intersection is not empty even in the context of singular problems,
	\item[c)] the characterization of the $\lambda$-behavior about continuity and monotonicity of the energy functional along the solutions  is new as well,
	\item[d)] Theorem \ref{TP1} and Theorem \ref{TP2} induce us to conjecture that there exists a bifurcation point $\tilde{\lambda}>0$ with $\lambda_{\ast} + \epsilon \leq \tilde{\lambda} \leq \lambda^{\ast}$ for which the two solutions collapse. 
\end{itemize}
Summarizing our results in a picture we have
\begin{figure}[h]
	\centering
	\begin{tikzpicture}[scale=.50]
	\draw[thick, ->] (-1, 3) -- (10, 3);
	\draw[thick, ->] (0, -1) -- (0, 7);
	\draw[thick] (0, 3) .. controls (0, 1) and (1, 1) ..(5,0.2);
	\draw (10,3) node[below]{$\lambda$};
	\draw (0, 3) node[below left]{$0$};
	\draw (0,7) node[left]{$Energy$};
	
	\draw[thick, dotted] (5,0.2) -- (8,-0.2);
	\draw (5,3) node[above]{$\lambda_{\ast}+\epsilon$};
	\draw (3,3) node[above]{$\lambda_{\ast}$};
	\draw (8,3) node[above]{$\tilde{\lambda} $};
	\draw (9,3) node[above]{$\lambda^{\ast} $};
	\draw (1.6,3) node[above]{$\hat{\lambda}$};
	\draw[thick] (0, 5) .. controls (0, 5) and (1, 2) ..(5,1);
	\draw[thick, dotted] (5, 1) -- (8,-0.2);
	\draw[thick, dotted] (5, 3) -- (5,0.2);
	\draw[thick, dotted] (8, 3) -- (8,-0.2);
	\draw (1.3,0.8) node[below]{$\Phi_{\lambda}(u_{\lambda})$};
	\draw (1.7,5) node[below]{$\Phi_{\lambda}(w_{\lambda})$};
	\draw (3,-1) node[below]{\textrm{Fig. 1  Energy depending on $\lambda$ }};
	\end{tikzpicture}
	\hspace{0.8cm}
\end{figure}

This paper is organized in the following way. In Section 2, we collect some technical results about the energy functional $\Phi_{\lambda},\lambda>0$ and the Nehari manifold associated to $\Phi_\lambda$. In Section 3, we show the multiplicity of solutions for $0<\lambda<\lambda_{{\ast}}$. The proof of multiplicity of solutions for $\lambda=\lambda_{{\ast}}$ will be done in Section 4 while in Section 5 we prove the multiplicity for $\lambda>\lambda_{{\ast}}$. In the Section 6 we prove Theorem \ref{TP1} and \ref{TP2}.

Throughout this paper, we make use of the following notations:
\begin{itemize}
	\item $c$ and $C$ are possibly different positive constants which may change from line to line,
	\item  $b^{+}=\max \left\{b,0\right\} $ is the positive part of the function $b$,
	\item  $S=\left\{u\in X:||u||=1 \right\} $ is the unitary sphere, where
	$$||u||^2= \int_{\mathbb{R}^{N}}\vert\nabla u \vert^2+V(x)u^2dx,$$
	\item $\langle \Phi^{'}(u), \psi \rangle $ denotes the G\^ateaux derivative of $\Phi$ at $u$ with respect to the direction $\psi \in X$. 
\end{itemize}
\section{Topological structures associated to the energy functional}
\mbox{}

In this section, let us assume the hypotheses of Theorem \ref{TP1} to prove some topological properties for the functional $\Phi_{\lambda}$. Let us endow $X$ with the inner product 
$$(u,w)= \int_{\mathbb{R}^{N}}\nabla u \nabla w+V(x)u wdx,$$
which turns $X$  into a Hilbert space with induced norm given by $||u||^{2}=(u,u)$. As a consequence, one deduces immediately from  $(V)_{0}$ that $X$ is embedded continuously into $H^{1}(\mathbb{R}^{N})$. 
\begin{lemma}
	\label{lemma1}
	The subspace $X$ is continuously embedded  into $H^{1}(\mathbb{R}^{N})$ for $q\in [2,2^{\ast}]$ and compact embedded for all $q\in [2,2^{\ast})$.
\end{lemma} 
It follows from Lemma \ref{lemma1} that
\begin{lemma}
	\label{lemma2}
	If $\lambda>0$ then $\Phi_{\lambda}$  is a  continuous and weakly lower semicontinuous 
	functional.
\end{lemma}
\begin{proof} We prove that $\Phi_{\lambda}$ is weakly lower semicontinuous (the proof of the continuity is almost similar). Take $\left\{u_{n}\right\}\subset X$ such that $u_{n}\rightharpoonup u$. It follows from Lemma \ref{lemma1} that 
	$$u_{n}\rightarrow u~\mbox{in}~ L^{q}(\mathbb{R}^{N}),~u_{n}\rightarrow u ~\mbox{a.e. in}~ \mathbb{R}^{N}~\mbox{and} ~|u_{n}(x)|\leq g_{q}(x)~\mbox{a.e. in}~\mathbb{R}^{N}.$$
	for some  $ g_{q}\in L^{q}(\mathbb{R}^{N})$. Since $0<\gamma < 1$, we obtain
	$$||u_{n}|^{1-\gamma}-|u|^{1-\gamma}|^{\frac{2}{1-\gamma}}\rightarrow 0~\mbox{and}~||u_{n}|^{1-\gamma}-|u|^{1-\gamma}|^{\frac{2}{1-\gamma}}\leq 2^{\frac{2}{1-\gamma}}g^{2}_{2}\in L^{1}(\mathbb{R}^{N})~\mbox{a.e. in}~\mathbb{R}^{N}.$$
	From $a\in L^{2/(1+\gamma)}(\mathbb{R}^N)$, the H\"older inequality and the Lebesgue dominated convergence theorem, we conclude that  
	$$\big| \int_{\mathbb{R}^{N}}a(x)(|u_{n}|^{1-\gamma}-|u|^{1-\gamma})\big|\leq [\int_{\mathbb{R}^{N}}(a(x))^{\frac{2}{1+\gamma}}]^{\frac{1+\gamma}{2}}[\int_{\mathbb{R}^{N}}||u_{n}|^{1-\gamma}-|u|^{1-\gamma}|^{\frac{2}{1-\gamma}}]^{\frac{1-\gamma}{2}}\rightarrow 0,$$

	Again, by using Lemma \ref{lemma1} and $b\in L^{\infty}(\mathbb{R}^N)$, we have  that $\int_{\mathbb{R}^{N}}b(x)|u_{n}|^{p+1}\rightarrow \int_{\mathbb{R}^{N}}b(x)|u|^{p+1}$ holds which completes the proof.
\end{proof}

Since we are interested in positive solutions, let us constrain  $\Phi_\lambda$ to the cone of non-negative functions of $X$, that is, 
$$X_{+}=\left\{u\in X\setminus \left\{0\right\}:u\geq 0 \right\}.$$
Define the $C^{\infty}$-fiber map $\phi_{\lambda,u}:(0,\infty)\rightarrow \mathbb{R}$ by
$$\phi_{\lambda,u}(t)=\Phi_{\lambda}(tu)=\frac{t^{2}}{2}||u||^{2}-\frac{t^{1-\gamma}\lambda}{1-\gamma}\int_{\mathbb{R}^{N}}a(x)|u|^{1-\gamma}dx-\frac{t^{p+1}}{p+1}\int_{\mathbb{R}^{N}}b(x)|u|^{p+1}dx,$$
for each $u\in X_{+}$ and $\lambda >0$ given. It is clear that
$$
\phi^{'}_{\lambda,u}(t)=t||u||^{2}-t^{-\gamma}\lambda \int_{\mathbb{R}^{N}}a(x)|u|^{1-\gamma}dx-t^{p}\int_{\mathbb{R}^{N}}b(x)|u|^{p+1}dx,
$$
\begin{equation}
	\label{17}
	\phi^{''}_{\lambda,u}(t)=||u||^{2}+\gamma t^{-\gamma-1}\lambda \int_{\mathbb{R}^{N}}a(x)|u|^{1-\gamma}dx-pt^{p-1}\int_{\mathbb{R}^{N}}b(x)|u|^{p+1}dx
\end{equation}
and if $u \in X_+$ is a solution of $(P_{\lambda})$, then $u\in \mathcal{N}_{\lambda} $, where
\begin{align*}
	\mathcal{N}_{\lambda}\equiv\left\{u\in X_{+}:||u||^{2}-\int_{\mathbb{R}^{N}}a(x)|u|^{1-\gamma}dx-\lambda\int_{\mathbb{R}^{N}}b(x)|u|^{p+1}dx=0\right\}
	=\left\{u\in X_{+}:\phi^{'}_{\lambda,u}(1)=0.\right\}.
\end{align*}
Although $\mathcal{N}_{\lambda}$ does not have enough regularity, let us refer to it as the Nehari manifold associated to $(P_{\lambda})$ from now on. It is classical to split it in three disjoint sets 
\begin{align*}
	\mathcal{N}^{-}_{\lambda}\equiv\left\{u\in  \mathcal{N}_{\lambda} :||u||^{2}+\gamma \lambda \int_{\mathbb{R}^{N}}a(x)|u|^{1-\gamma}dx-p\int_{\mathbb{R}^{N}}b(x)|u|^{p+1}dx<0\right\}=\left\{u\in  \mathcal{N}_{\lambda}:\phi^{''}_{\lambda,u}(1)<0\right\},~~~~~~~~~~~~~~~~~~~~~~~~~~~~~~~~~~~~~~~~~~~~~~~~~~~~~~~ \\
	\mathcal{N}^{+}_{\lambda}\equiv\left\{u\in  \mathcal{N}_{\lambda} :||u||^{2}+\gamma \lambda\int_{\mathbb{R}^{N}}a(x)|u|^{1-\gamma}dx-p\int_{\mathbb{R}^{N}}b(x)|u|^{p+1}dx>0\right\}
	=\left\{u\in  \mathcal{N}_{\lambda}:\phi^{''}_{\lambda,u}(1)>0\right\},~~~~~~~~~~~~~~~~~~~~~~~~~~~~~~~~~~~~~~~~~~~~~~~~~~~~~~~\\                                                             \mathcal{N}^{0}_{\lambda}\equiv\left\{u\in  \mathcal{N}_{\lambda} :||u||^{2}+\gamma \lambda \int_{\mathbb{R}^{N}}a(x)|u|^{1-\gamma}dx-p\int_{\mathbb{R}^{N}}b(x)|u|^{p+1}dx=0\right\}=\left\{u\in  \mathcal{N}_{\lambda}:\phi^{''}_{\lambda,u}(1)
	=0\right\}. ~~~~~~~~~~~~~~~~~~~~~~~~~~~~~~~~~~~~~~~~~~~~~~~~~~~~~~~
	& 
\end{align*}
We will study the  structure of the sets $\mathcal{N}^{-}_{\lambda}, \mathcal{N}^{0}_{\lambda}, \mathcal{N}^{+}_{\lambda}$ and show existence of solutions on $\mathcal{N}^{-}_{\lambda}$ and $\mathcal{N}^{+}_{\lambda}$. The easiest case is when $\mathcal{N}_\lambda^0=\emptyset$. One of our main contributions to the literature of singular problems is to show existence of solutions on $\mathcal{N}^{-}_{\lambda}$ and $\mathcal{N}^{+}_{\lambda}$ beyond the extremal value, for which  $\mathcal{N}^{0}_{\lambda}$ is not empty anymore. 

The next proposition is straightforward.
\begin{proposition} \label{prop21} Let $u\in X_{+}$ and $\lambda>0$. If $\int b|u|^{p+1}\leq 0$, then $\phi_{\lambda,u}$  has only one critical point at $t^{+}_{\lambda}(u)\in (0,\infty),$ which satisfies $\phi^{''}_{\lambda,u}(t^{+}_{\lambda}(u))>0$. If $\int b|u|^{p+1}>0$, then there are three possibilities:
	\begin{itemize}
		\item[$(I)$] there are only two critical points for $\phi_{\lambda,u}$. The first one is $t^{+}_{\lambda}(u)$ with $\phi^{''}_{\lambda,u}(t^{+}_{\lambda}(u))>0$ and the second one is $t^{-}_{\lambda}(u)$ with $\phi^{''}_{\lambda,u}(t^{-}_{\lambda}(u))<0$. Moreover, $\phi_{\lambda,u}$ is decreasing over the intervals $[0,t^{+}_{\lambda}(u)], [t^{-}_{\lambda}(u),\infty)$ and increasing over the the interval $[t^{+}_{\lambda}(u),t^{-}_{\lambda}(u)]$ $($evidently $0<t^{+}_{\lambda}(u)<t^{-}_{\lambda}(u)$ \!$)$,
		\item[$(II)$] there is only one critical point $t^{0}_{\lambda}(u)>0$ for $\phi_{\lambda,u}$, which is an inflection point. Moreover, $\phi_{\lambda,u}$ is decreasing for $t>0$,
		\item[$(III)$] the function $\phi_{\lambda,u}$ is decreasing for $t>0$ and has no critical points.
	\end{itemize}
\end{proposition}
Let us study the set $\mathcal{N}_\lambda^0$. One can easily see that if $u\in\mathcal{N}_\lambda^0$ then $\int_{\mathbb{R}^{N}}b|u|^{p+1}>0$, therefore, we introduce the set 
 $$
 Z^{+}\equiv\left\{u\in X_{+}:\int_{\mathbb{R}^{N}}b|u|^{p+1}>0 \right\}.
 $$
Observe that $Z^+$ is a cone. For $u\in Z^+$ coonsider the system
$$
\begin{array}{c}
\phi^{'}_{\lambda,u}(t)=
\phi^{''}_{\lambda,u}(t)=0,
\end{array}
$$
that is
$$
\left\{
\begin{aligned}
t||u||^{2}-t^{-\gamma}\lambda \int_{\mathbb{R}^{N}}a(x)|u|^{1-\gamma}dx-t^{p}\int_{\mathbb{R}^{N}}b(x)|u|^{p+1}dx=0,
\\
||u||^{2}+\gamma \lambda t^{-\gamma-1}\int_{\mathbb{R}^{N}}a(x)|u|^{1-\gamma}dx-pt^{p-1}\int_{\mathbb{R}^{N}}b(x)|u|^{p+1}dx=0.
\end{aligned}
\right.
$$
The system has a unique solution which is given by $(t(u),\lambda(u))$, where
\begin{equation}\label{E2}
	\left\{
	\begin{aligned}
		t(u)&= \left(\frac{1+\gamma}{p+\gamma}\right)^{\frac{1}{p-1}} \left[\frac{||u||^{2}}{\int_{\mathbb{R}^{N}} b|u|^{p+1}}\right]^{\frac{1}{p-1}}
		\\
		\lambda(u)&=C(\gamma,p)\frac{\left(||u||^{2}\right)^{\frac{p+\gamma}{p-1}}}{\left[\int_{\mathbb{R}^{N}}b|u|^{p+1}\right]^{\frac{1+\gamma}{p-1}}\left[\int_{\mathbb{R}^{N}}a|u|^{1-\gamma}\right]},
	\end{aligned}
	\right.
\end{equation}
where
$$C(\gamma,p)\equiv\left(\frac{1+\gamma}{p+\gamma}\right)^{\frac{1+\gamma}{p-1}}\left(\frac{p-1}{p+\gamma}\right).$$
From the definition of $\lambda(u)$ we conclude that
\begin{proposition}\label{lambdava} Suppose that $u\in Z^{+}$. Then, if $\lambda\in(0,\lambda(u))$ the fiber map $\phi_{\lambda,u}$ satisfies $I)$ of Proposition $\ref{prop21}$, while $\phi_{\lambda(u),u}$ satisfies $II)$ and if $\lambda\in (\lambda(u),\infty)$ it must satisfies $III)$.
\end{proposition}
Define
$$
\lambda_{\ast}=\displaystyle \inf_{u\in Z^{+}}\lambda(u).
$$
\begin{lemma}\label{L0} The function $\lambda$ defined in \eqref{E2} is continuous, $0$-homogeneous and unbounded from above. Moreover, $\lambda_{\ast}>0$ and there exists $u\in Z^{+}$ such that $\lambda_{\ast} = \lambda(u)$.
\end{lemma}
\begin{proof} The continuity and $0$-homogeneity are obvious. From these properties, it follows that the rest of the proof can be done by considering $\lambda$ restricted to the set $Z^+\cap S$, where $S=\{u\in X:\ \|u||=1\}$. To prove that $\lambda$ is unbounded from above, first note that the functional $F_b:X\longrightarrow \mathbb{R}$ defined by $F_b(u)=\int_{\mathbb{R}^{N}}b|u|^{p+1}$ is continuous and therefore $F_b^{-1}((0,\infty))\cap S$ is a open set in $S$. Moreover, since $F_b(tu)=t^{p+1}F_b(u)$ for $t>0$, it follows that $F_b^{-1}((0,\infty))\cap S\neq S$ and therefore there exists a sequence $\left\{u_{n}\right\}\subset F_b^{-1}((0,\infty))\cap S$ such that $F_b(u_{n})\to 0$ in $X$. Consequently
	\begin{equation*}
	\lim_{n\to \infty}\lambda(u_n)=\lim_{n\to \infty}\frac{C(\gamma,p)}{\left[\int_{\mathbb{R}^{N}}b|u_{n}|^{p+1}\right]^{\frac{1+\gamma}{p-1}}\left[\int_{\mathbb{R}^{N}}a|u_{n}|^{1-\gamma}\right]}= \infty,
	\end{equation*}
	which proves that $\lambda$ is unbounded from above. Now observe that 
	$$\lambda_{\ast}=\displaystyle \inf_{u\in Z^{+}\cap S}\lambda(u)\geq cC(\gamma,p)\Vert a \Vert_{2/(1+\gamma)}^{-1}\Vert b \Vert_{\infty}^{-1}>0$$
	for some $c>0$. To end the proof,  
	take $\displaystyle \left\{u_{n}\right\}\subset Z^{+}\cap S$ such that $\lambda(u_{n})\rightarrow \lambda_{\ast}$. So, it follows from Lemma \ref{lemma1} that 
	$$
	\begin{array}{c}
	u_{n}\rightharpoonup u\in X, ~~
	u_{n}\rightarrow u ~\mbox{in}~ L^{q}(\mathbb{R}^{N})~\mbox{for each}~ q\in [2,2^{\ast})~~\mbox{and}~~
	u_{n}(x)\rightarrow u(x) ~\mbox{a.e. in}~\mathbb{R}^{N},
	\end{array}
	$$
	which lead us to infer that $u\not\equiv 0$. Otherwise, we would have 
	\begin{equation}
		\label{19}
		\lambda_{\ast}=\lim_{n\to \infty}\lambda(u_{n})=\lim_{n\to \infty}\frac{C(\gamma,p)}{\left[\int_{\mathbb{R}^{N}}b|u_{n}|^{p+1}\right]^{\frac{1+\gamma}{p-1}}\left[\int_{\mathbb{R}^{N}}a|u_{n}|^{1-\gamma}\right]}= \infty,
	\end{equation}
	which is an absurd. Let $v=\frac{u}{||u||}\in X_{+}\cap S$. If  $u_{n}\nrightarrow u$ in $X$, it would follow by  the weak lower semi-continuity of the norm that
	$$\lambda(v)=\lambda\left(\frac{u}{\|u\|}\right)=\lambda(u)<\liminf \lambda(u_{n})=\lambda_{\ast},$$
	but this is impossible. It follows that  $u\in Z^{+}\cap S$ and $\lambda(u)=\lambda_{\ast}$. This ends the proof. \fim
\end{proof} 

\newpage

 Proposition \ref{prop21} and Lemma \ref{L0} are described on the following pictures:

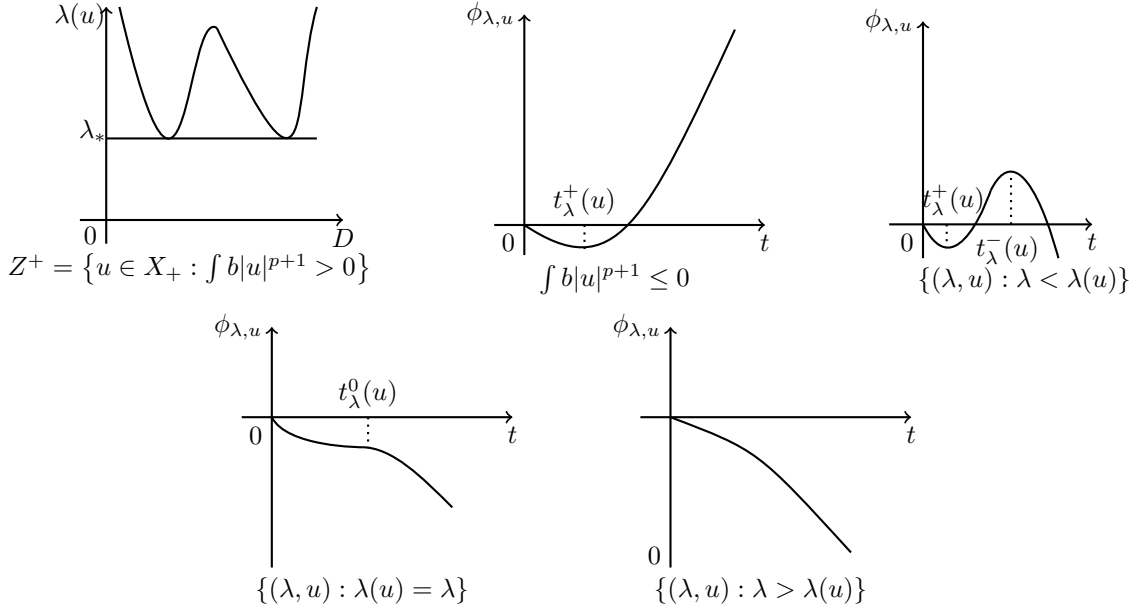
\begin{figure}[h]\label{F1}
	\centering
	\begin{tikzpicture}[scale=.35]
	\draw[thick, ->] (-1, -3.1) -- (9, -3.1);
	\draw[thick, ->] (0,-4) -- (0, 5);
	
	\draw[thick] (0.5, 5) .. controls (3.1, -6.2) and (3, 5) ..(4.2,4.2);
	\draw[thick] (4.2,4.2).. controls (8, -3.8) and (7, 1.5) ..(8,5);
	\draw (9,-3.1) node[below]{$D$};
	\draw (0, -3) node[below left]{$0$};
	\draw (-0.5,-0.5) node[above]{$\lambda_{\ast}$};
	\draw[thick] (0, 0) -- (8,0);
	\draw (-1,5.5) node[below]{$\lambda(u)$};
	\draw (3,-4) node[below]{\textrm{ $Z^{+}=\left\{u\in X_{+}:\int b\vert u\vert^{p+1}>0\right\}$}};
	\end{tikzpicture}
	\hspace{0.8cm}
	\begin{tikzpicture}[scale=.40]
	figura
	\draw[thick, ->] (-1, 0) -- (8, 0);
	\draw[thick, ->] (0, -1) -- (0, 7);
	\draw[thick] (0, 0) .. controls (3, -2) and (4, 0) ..(7,6.5);
	\draw (8,0) node[below]{$t$};
	\draw (0, 0) node[below left]{$0$};
	\draw (0,7) node[left]{$\phi_{\lambda,u}$};
	\draw (2,0) node[above]{$t^{+}_\lambda (u)$};
	\draw[thick, dotted] (2,0) -- (2,-0.8);
	\draw (3,-1) node[below]{\textrm{$\int b\vert u\vert ^{p+1}\leq 0$}};
	\end{tikzpicture}
	\hspace{0.8cm}
	\begin{tikzpicture}[scale=0.45]
	\draw[thick, ->] (-1, 0) -- (5, 0);
	\draw[thick, ->] (0, -1) -- (0, 6);
	\draw[thick] (0, 0) .. controls (1, -2) and (2, 1) ..(2,1);
	\draw[thick] (2, 1) .. controls (3, 3) and (4,-1) ..(4,-1);
	\draw[thick, dotted] (0.7,0) -- (0.7,-0.6);
	\draw (5,0) node[below]{$t$};
	\draw (0, 0) node[below left]{$0$};
	\draw (0.9,0) node[above]{$t^{+}_{\lambda}(u)$};
	\draw (2.5,0) node[below]{$t^{-}_{\lambda}(u)$};
	\draw[thick, dotted] (2.6,0) -- (2.6,1.5);
	\draw (0,6) node[left]{$\phi_{\lambda,u}$};
	\draw (3,-1) node[below]{\textrm{$\left\{(\lambda,u):\lambda<\lambda(u)\right\}$}};
	\end{tikzpicture}

\begin{tikzpicture}[scale=0.40]
	\draw[thick, ->] (-1, 4) -- (8, 4);
	\draw[thick, ->] (0, -1) -- (0, 7);
	\draw[thick] (0 ,4) .. controls (0.5, 3) and (3,3) .. (3, 3) ;
	\draw[thick] (3 ,3) .. controls (4, 3) and (5,2) .. (6, 1) ;
	\draw[thick, dotted] (3.2,4) -- (3.2, 3);
	\draw (8,4) node[below]{$t$};
	\draw (0, 4) node[below left]{$0$};
	\draw (0,7) node[left]{$\phi_{\lambda,u}$};
	\draw (3.2,4) node[above]{$t^{0}_{\lambda}(u)$};
	\draw (3,-1) node[below]{\textrm{$\left\{(\lambda,u):\lambda(u)=\lambda\right\}$}};
	\end{tikzpicture}	
	\hspace{0.8cm}
	\begin{tikzpicture}[scale=0.40]
	\draw[thick, ->] (-1, 4) -- (8, 4);
	\draw[thick, ->] (0, -1) -- (0, 7);
	\draw[thick] (0 ,4) .. controls (2.9,2.9) and (2.9,2.9) .. (6,-0.5);
	\draw (8,4) node[below]{$t$};
	\draw (0, 0) node[below left]{$0$};
	\draw (0,7) node[left]{$\phi_{\lambda,u}$};
	\draw (3,-1) node[below]{\textrm{$\left\{(\lambda,u):\lambda>\lambda(u)\right\}$}};
	\end{tikzpicture}
\end{figure}

From Proposition \ref{prop21} and Lemma \ref{L0} we obtain

\begin{lemma}\label{L1} For each $\lambda>0$ we have that $\mathcal{N}^{+}_{\lambda}, \mathcal{N}^{-}_{\lambda}\neq \emptyset$. Moreover:
	\begin{itemize}
		\item[$a)$] $\mathcal{N}^{0}_{\lambda}= \emptyset$ for $0<\lambda <\lambda_{\ast}$,
		\item[$b)$]  $\mathcal{N}^{0}_{\lambda}\neq \emptyset$ for $\lambda \geq\lambda_{\ast}$.
	\end{itemize}
\end{lemma}
\begin{proof}
 First we will to prove that $\mathcal{N}^{+}_{\lambda}, \mathcal{N}^{-}_{\lambda}\neq \emptyset$. By Lemma \ref{L0} for each $\lambda>0$ there exists $u\in Z^{+}$ such that $\lambda <\lambda(u)$. Thus by Proposition \ref{lambdava} there exists $t^{+}_{\lambda}(u)< t^{-}_{\lambda}(u)$ such that $t^{+}_{\lambda}(u)u\in \mathcal{N}^{+}_{\lambda}$ and $t^{-}_{\lambda}(u)u\in \mathcal{N}^{-}_{\lambda}$. Hence $\mathcal{N}^{+}_{\lambda}\neq \emptyset, \mathcal{N}^{-}_{\lambda}\neq \emptyset$.

To prove $a)$ we first note that if $u\in Z^{+}$ then from Lemma \ref{L0} there holds $\lambda(u)\geq \lambda_{\ast}$. Hence, if $\lambda \in (0,\lambda_{\ast})$ it follows from Proposition \ref{lambdava} that $u\notin \mathcal{N}^{0}_{\lambda}$. If $u\notin Z^+$, then $\int_{\mathbb{R}^{N}}b|u|^{p+1}\leq 0$ and by Proposition \ref{prop21}, $\phi_{\lambda,u}$ has only one critical point at $t^{+}_{\lambda}(u)\in (0,\infty)$, which satisfies $\phi^{''}_{\lambda,u}(t_{\lambda}^{+}(u))>0$ which implies again that $u\notin \mathcal{N}^{0}_{\lambda}$. Therefore $\mathcal{N}^{0}_{\lambda}= \emptyset$ for $0<\lambda <\lambda_{{\ast}}$.

Now we prove $b)$. Indeed, from the definition of $\lambda(u)$ we know that 
\begin{equation*}
t(u)u\in \mathcal{N}_{\lambda(u)}^0.
\end{equation*}
From Lemma \ref{L0} we know that for each $\lambda\ge \lambda^*$, there exits $u\in Z^+$ such that $\lambda(u)=\lambda$ which ends the proof.

\fim
	\end{proof}
 
Now we characterize the Nehari set $\mathcal{N}_{\lambda_\ast}^0$. Note that the singular term forces the non-differentiability of the function $\lambda(u)$ at some points, however, at the global minimum points we prove that it has null derivative. 
\begin{lemma}\label{L2} There holds
	\begin{equation}
		\label{109}
		\mathcal{N}^{0}_{\lambda_{\ast}}=\displaystyle \left\{u\in \mathcal{N}_{\lambda_{\ast}}:\int_{\mathbb{R}^{N}}b|u|^{p+1}> 0,\lambda(u)=\lambda_{\ast}\right\},
	\end{equation}
	and
	\begin{equation}\label{N1}
		(u,\psi)-(p+1)\displaystyle \int_{\mathbb{R}^{N}}b(x)u^{p}\psi dx-(1-\gamma)\lambda_{\ast}\displaystyle \int_{\mathbb{R}^{N}}a(x)u^{-\gamma}\psi dx=0,~\forall \psi \in X,
	\end{equation}
	holds for each $u\in \mathcal{N}^{0}_{\lambda_{\ast}}$ given.
\end{lemma}
\begin{proof} The characterization of $\mathcal{N}^{0}_{\lambda_{\ast}}$ is a consequence of Lemma \ref{L0}. Let us prove (\ref{N1}) by splitting the proof in three steps. First, let rewrite the function $\lambda(u)$ as $\lambda(u)=C(\gamma,p)f(u)g(u)$, where
	$$f(u)=\frac{1}{\int_{\mathbb{R}^{N}}a|u|^{1-\gamma}dx}~~\mbox{and}~~g(u)=\frac{\left(||u||^{2}\right)^{\frac{p+\gamma}{p-1}}}{\left[\int_{\mathbb{R}^{N}}b|u|^{p+1}\right]^{\frac{1+\gamma}{p-1}}}.$$
	\noindent\textbf{Step.1.} $\langle f^\prime(u),\psi \rangle$ there exists for all $\psi \in X_{+}$ and for all $u\in \mathcal{N}^{0}_{\lambda_{\ast}}$.
	
	In fact, for such $u,\psi $ given, it follows by continuity that  $\int_{\mathbb{R}^{N}}b|u+t\psi|^{p+1}>0$ for  $t>0$ small enough. Therefore $g(u+t\psi)$ is well defined for $t>0$ small enough and $\langle  g^{\prime}(u), \psi \rangle $ there exists. Since, $u$ is the minimum point for $\lambda(u)$, we have that
	$\lambda(u+t\psi)-\lambda(u)=\lambda(u+t\psi)-\lambda_{\ast}\geq 0,~\forall ~t\geq 0$ enough small, that  implies 
	$$
	(g(u+t\psi)-g(u))f(u+t\psi)\geq -g(u)(f(u+t\psi)-f(u)).
	$$
	
	Since, 
	$$f(u+t\psi)-f(u)=-h(t)^{-2}\left[\int_{\mathbb{R}^{N}}a|u+t\psi|^{1-\gamma} dx-\int_{\mathbb{R}^{N}}a|u|^{1-\gamma} dx\right],$$
	where 
	$$h(t)=\theta(t)\int_{\mathbb{R}^{N}}a|u+t\psi|^{1-\gamma} dx+(1-\theta(t))\int_{\mathbb{R}^{N}}a|u|^{1-\gamma} dx,~\theta(t)\in [0,1],$$ 
	is a measurable function such that $h(t)\rightarrow h(0)=\int_{\mathbb{R}^{N}}a|u|^{1-\gamma} dx\neq 0$ with $t\rightarrow 0^{+}$, it follows from Fatou's lemma, that
	\begin{align*}
		\infty> \langle  g^{\prime}(u), \psi \rangle f(u)\geq g(0)
		\left[\int_{\mathbb{R}^{N}}a|u|^{1-\gamma} dx\right]^{-2}\displaystyle \liminf_{t\to 0^{+}} \int_{\mathbb{R}^{N}}\frac{a|u+t\psi|^{1-\gamma}-a|u|^{1-\gamma}}{t}  & \\
		\geq g(0) \left[\int_{\mathbb{R}^{N}}a|u|^{1-\gamma} dx\right]^{-2}(1-\gamma)\int_{\mathbb{R}^{N}}aG(x)\psi dx, ~~~~~~~~~~~~~~~& 
	\end{align*}
	where
	$$
	G(x)=\left\{
	\begin{array}{rcl}
	u^{-\gamma}(x), & \mbox{if} & u(x)\neq 0 ,\\
	\infty, & \mbox{if} & u(x)=0.
	\end{array}
	\right.
	$$
	
	So, by taking $\psi>0,~\psi\in X$ above, we obtain that $G(x)=u^{-\gamma}(x)$ for all $x\in \mathbb{R}^{N}$, that is, $u>0$ in $\mathbb{R}^{N}$. This implies that $0< \int_{\mathbb{R}^{N}}au^{-\gamma} \psi dx<\infty$ for all $\psi\in X_{+}$. As a consequence, we have $\langle  j^{\prime}(u), \psi \rangle$ there exists, where $j(u)=\int_{\mathbb{R}^{N}}a|u|^{1-\gamma} dx,~\psi\in X_{+}$. To end the proof, we just note that $f(u)=\left[j(u)\right]^{-1}$ and hence
	$$\langle  f^{\prime}(u), \psi \rangle=-(1-\gamma)\left[\int_{\mathbb{R}^{N}}a|u|^{1-\gamma} dx\right]^{-2}\int_{\mathbb{R}^{N}}au^{-\gamma} \psi dx$$
	holds. 
	
	Before proving (\ref{N1}), let us prove the Step 2 by assuming without loss of generality that $||u||=1$.
	
	\noindent\textbf{Step.2.}  There holds
	\begin{equation}
		\label{111}
		2(u,\psi)-(p+1)\displaystyle \int_{\mathbb{R}^{N}}b(x)u^{p}\psi dx-(1-\gamma)\lambda_{\ast}\displaystyle \int_{\mathbb{R}^{N}}a(x)u^{-\gamma}\psi dx\geq 0,~\forall \psi\in X_{+}.
	\end{equation}
	
	Indeed, since $u \in X$ is minimum point of $\lambda(u)$ such that $\int b\vert u \vert^{p+1}>0$, we have
	\begin{align}\label{112}
		\left\{\frac{\left(\frac{2(p+\gamma)}{p-1}\right)(u,\psi)\displaystyle \left[F(u)\right]^{\frac{1+\gamma}{p-1}}-(p+1)\left(\frac{1+\gamma}{p-1}\right)\left[F(u)\right]^{\frac{2+\gamma-p}{p-1}}\displaystyle \int_{\mathbb{R}^{N}}b(x)u^{p}\psi dx}{\left[F(u)\right]^{\frac{2(1+\gamma)}{p-1}}}\right\} \left[H(u)\right]^{-1} 
		\\-(1-\gamma)\displaystyle\frac{\left[H(u)\right]^{-2}\left[\displaystyle\int_{\mathbb{R}^{N}}a(x)u^{-\gamma}\psi dx\right]}{\left[F(u)\right]^{\frac{1+\gamma}{p-1}}} \geq 0,\nonumber
	\end{align}
	for all $\psi\in X_{+}$, where
	\begin{equation}
		\label{115}
		F(u)=\int_{\mathbb{R}^{N}}b(x)u^{p+1}dx~\mbox{and}~ H(u)=\int_{\mathbb{R}^{N}}au^{1-\gamma}dx.
	\end{equation}
	
	Once using that $u\in \mathcal{N}^{0}_{\lambda_{\ast}}$, we are able to infer that
	$$ H(u)=\int_{\mathbb{R}^{N}}au^{1-\gamma}dx=\frac{p-1}{\lambda_{\ast}(p+\gamma)}~\mbox{and}~ F(u)=\int_{\mathbb{R}^{N}}b(x)u^{p+1}dx=\frac{1+\gamma}{p+\gamma}.$$
	
	Thus, by using these expressions in (\ref{112}), we get (\ref{111}) after some manipulations.
	
	Finely, by using the characterization (\ref{109}) and adjusting an argument from Graham-Eagle \cite{GE}, we are able to show the equality (\ref{N1}).
	
	\noindent\textbf{Step.3.} There holds
	$$2(u,\psi)-(p+1)\displaystyle \int_{\mathbb{R}^{N}}b(x)u^{p}\psi dx-(1-\gamma)\lambda_{\ast}\displaystyle \int_{\mathbb{R}^{N}}a(x)u^{-\gamma}\psi dx=0,~\forall~ \psi\in X.$$
	To do this, let us set $\Psi:=(u+\epsilon \psi)^{+}\in X_{+}$ for $\epsilon>0$. Since (\ref{111}) holds, it follows from splitting the whole space in $\left\{u+\epsilon \psi>0\right\}$ and $\left\{u+\epsilon \psi\leq 0\right\}$, that
	\begin{align}
		\label{118}
		0\leq 2(u,\Psi)-(p+1)\displaystyle \int_{\mathbb{R}^{N}}b(x)u^{p}\Psi dx-(1-\gamma)\lambda_{\ast}\displaystyle \int_{\mathbb{R}^{N}}a(x)u^{-\gamma}\Psi dx& \nonumber\\
		=2||u||^{2}-(p+1)\displaystyle \int_{\mathbb{R}^{N}}b(x)u^{p+1}dx-(1-\gamma)\lambda_{\ast}\displaystyle \int_{\mathbb{R}^{N}}a(x)u^{1-\gamma}dx& \\
		+\epsilon\left[\int_{\mathbb{R}^{N}}2(\nabla u\nabla \psi+V(x)u \psi)-(p+1)b(x)u^{p}\psi-(1-\gamma)\lambda_{\ast}a(x)u^{-\gamma} \psi)\right]\nonumber\\
		-2\int_{\left\{u+\epsilon \psi\leq 0\right\}}(|\nabla u|^{2}+V(x)u^{2})+(p+1)\int_{\left\{u+\epsilon \psi\leq 0\right\}}b(x)u^{p}(u+\epsilon \psi)\nonumber\\
		+(1-\gamma)\lambda_{\ast}\int_{\left\{u+\epsilon \psi\leq 0\right\}}au^{-\gamma}(u+\epsilon \psi)dx -2\epsilon \int_{\left\{u+\epsilon \psi\leq 0\right\}}(\nabla u\nabla \psi+V(x)u \psi). \nonumber
	\end{align} 	
	
	Now, by using $0< \gamma <1$ and again splitting $\left\{u+\epsilon \psi\leq 0\right\}$ in $\left\{u+\epsilon \psi\leq 0\right\} \cap \{b<0\}$ and $\left\{u+\epsilon \psi\leq 0\right\} \cap \{b\geq0\}$, we obtain
	\begin{align}\label{114}	
		0&\leq  2(u,\psi)-(p+1)\displaystyle \int_{\mathbb{R}^{N}}b(x)u^{p}\psi dx-(1-\gamma)\lambda_{\ast}\displaystyle \int_{\mathbb{R}^{N}}a(x)u^{-\gamma}\psi dx\nonumber\\
		&	\leq \epsilon\left[\int_{\mathbb{R}^{N}}2(\nabla u\nabla \psi+V(x)u \psi)-(p+1)b(x)u^{p}\psi-(1-\gamma)\lambda_{\ast}a(x)u^{-\gamma} \psi)\right]\\
		&-2\epsilon \int_{\left\{u+\epsilon \psi\leq 0\right\}}(\nabla u\nabla \psi+V(x)u \psi)+\epsilon (p+1)\int_{\left\{u+\epsilon \psi\leq 0,\left\{b<0\right\}\right\}}b(x)u^{p} \psi.\nonumber
	\end{align}
	
	Since the measure of the domains  of integration $\left\{u+\epsilon \psi\leq 0\right\}$ and $\left\{u+\epsilon \psi\leq 0\right\} \cap \{b<0\}$  tends to zero as $\epsilon\rightarrow 0$, we have
	from (\ref{114}) that
	
	\begin{align*}	
		0\leq & \int_{\mathbb{R}^{N}}(2(\nabla u\nabla \psi+V(x)u \psi)-(p+1)b(x)u^{p}\psi-(1-\gamma)\lambda_{\ast}a(x)u^{-\gamma} \psi)\\
		=&2(u,\psi)-(p+1)\displaystyle \int_{\mathbb{R}^{N}}b(x)u^{p}\psi dx-(1-\gamma)\lambda_{\ast}\displaystyle \int_{\mathbb{R}^{N}}a(x)u^{-\gamma}\psi dx
	\end{align*}
	holds. So, the equality is a consequence of taking  $-\psi$ in the above inequality. This ends the proof.
\end{proof}
\fim
\vspace{0.4cm}

The following result will be very important to show multiplicity of solutions to problem (\ref{pq})  at  $\lambda=\lambda_{\ast}$ and in particular it shows that these solutions belongs to $\mathcal{N}_{\lambda_{\ast}}^{-}$ and $\mathcal{N}_{\lambda_{\ast}}^{+}$, respectively.
\begin{corollary}\label{C1} The problem $(P_{\lambda_{\ast}})$ has no solution $u_{\lambda_{\ast}} \in \mathcal{N}_{\lambda_{\ast}}^{0}$.
	
\end{corollary}
\begin{proof}
	If there exists a solution $u_{\lambda_{\ast}}\in \mathcal{N}_{\lambda_{\ast}}^{0}$ for $(P_{\lambda_{\ast}})$, then it  would follows from  Lemma \ref{L2}-(\ref{N1}) that
	$$\displaystyle \int_{\mathbb{R}^{N}}[(p-1)b(x)u_{\lambda_{\ast}}^{p}-(1+\gamma)\lambda_{\ast}a(x)u_{\lambda_{\ast}}^{-\gamma}]\psi=0, \forall \psi\in X,$$
	that is,
	$$(p-1)b(x)u_{\lambda_{\ast}}^{p}(x)=(1+\gamma)\lambda_{\ast}a(x)u_{\lambda_{\ast}}^{-\gamma}(x)~\mbox{a.e. in}~\mathbb{R}^{N}.$$
	Therefore we have two possibilities. If $b(x)\leq 0$ in $\Omega\subset \mathbb{R}^{N}$ with $|\Omega|>0$, then $(1+\gamma)a(x)u_{\lambda_{\ast}}^{-\gamma}\leq 0$ in $\Omega$, which is an absurd. If $b>0$ in $\mathbb{R}^{N}$, then 
	$$u_{\lambda_{\ast}}=\left[\frac{a(x)\lambda_{\ast}(1+\gamma)}{b(x)(p-1)}\right]^{\frac{1}{p+\gamma}}\notin X,$$
	which is an absurd again.
\end{proof}
\fim
\vspace{0.4cm}

The following result will be essential in order to prove the existence of multiple  solutions for $\lambda >\lambda_{\ast}$ as well. Due to the presence of the singular term, the arguments used for regular cases, see for instance  Corollary 2 in (see \cite{SM}), does not work anymore. 

\begin{lemma}\label{L4}
	The set $\mathcal{N}^{0}_{\lambda_{\ast}}$ is compact.
\end{lemma}
\begin{proof}
	First, we note that $u\in \mathcal{N}^{0}_{\lambda_{\ast}}$  implies that
	$$(1+\gamma)||u||^{2}=(\gamma +p)\int_{\mathbb{R}^{N}}b(x)|u|^{p+1}dx ~\mbox{and}~ (p-1)||u||^{2}=\lambda_{\ast}(\gamma +p)\int_{\mathbb{R}^{N}}a(x)|u|^{1-\gamma}dx.$$
	Thus, by using the H\"older's inequality and the Sobolev embeddings $X\hookrightarrow L^{p+1}(\mathbb{R}^{N}),L^{2}(\mathbb{R}^{N})$, we obtain 
	\begin{equation}\label{E4}
		c\leq ||u||\leq C
	\end{equation}
	for some $c,C>0$.
	
	Set $\left\{u_{n}\right\}\subset \mathcal{N}^{0}_{\lambda_{\ast}}$. Thus we may assume that $u_{n}\rightharpoonup u\in X$ in $X$, $u_{n}\to u$ in $L^{q}(\mathbb{R}^{N})$ for $q\in [2,2^{\ast})$ and $u\geq 0$. This, together with (\ref{E4}), imply  that
	$$0<c\leq \liminf_{n \to \infty} ||u_{n}||^{2}=\left(\frac{(\gamma +p)}{1+\gamma}\right)\lim_{n \to \infty}\int_{\mathbb{R}^{N}}b(x)|u_{n}|^{p+1}dx=  \left(\frac{(\gamma +p)}{1+\gamma}\right) \int_{\mathbb{R}^{N}}b(x)|u|^{p+1}dx,$$
	that is, $u\not\equiv 0$. 
	
	Now, we claim that $u_{n}\to u$ in $X$. Indeed, if not,  it would follow from the continuities of $F$ and $H$ (see (\ref{115})), that 
	$$\lambda(u)=\left(\frac{1+\gamma}{p+\gamma}\right)^{\frac{1+\gamma}{p-1}}\left(\frac{p-1}{p+\gamma}\right)\frac{\left(||u||^{2}\right)^{\frac{p+\gamma}{p-1}}}{\left[\int_{\mathbb{R}^{N}}b|u|^{p+1}\right]^{\frac{1+\gamma}{p-1}}\left[\int_{\mathbb{R}^{N}}a|u|^{1-\gamma}\right]}<\liminf \lambda(u_{n})=\lambda_{{\ast}},$$
	which is an absurd, therefore, $u_{n}\to u$ in $X$ and consequently $\mathcal{N}^{0}_{\lambda_{\ast}}$ is compact.
	This ends the proof.
\end{proof}
\fim
\vspace{0.4cm}

Below, by taking advantage of Lemma \ref{L1}, we define for each $\lambda>0$ the non-empty set
$$\hat{\mathcal{N}}_{\lambda}=\left\{u\in X_{+}:\int_{\mathbb{R}^{N}}b|u|^{p+1}>0, \phi_{\lambda,u}~\mbox{has two critical points}\right\} ,$$
and the set
$$\hat{\mathcal{N}}_{\lambda}^{+}=\left\{u\in X_{+}:\int_{\mathbb{R}^{N}}b|u|^{p+1}\leq 0 \right\},$$
which may be empty.

Let $\overline{\hat{\mathcal{N}}_{\lambda}\cup \hat{\mathcal{N}}_{\lambda}^{+}}$ be the closure of $\hat{\mathcal{N}}_{\lambda}\cup \hat{\mathcal{N}}_{\lambda}^{+}$ with respect to the norm topology.  After a few modifications in the proofs of Propositions 2.9, 2.10 and Corollary 2.11 in \cite{SM}, we have
\begin{proposition}\label{r1} There holds:
	\begin{enumerate}
		\item[$(i)$] if $\lambda_{1},\lambda_{2}\in (0,\lambda_{\ast})$, then  $\hat{\mathcal{N}}_{\lambda_{1}}=\hat{\mathcal{N}}_{\lambda_{2}}$,
		\item[$(ii)$] if $u\in \hat{\mathcal{N}}_{\lambda}$, then $tu\in \hat{\mathcal{N}}_{\lambda}$ for all $t>0$, that is, $\hat{\mathcal{N}}_{\lambda}$ is a positive cone generated by the  set $\mathcal{N}_{\lambda}^{+}\cup \mathcal{N}_{\lambda}^{-}$.  More specifically, 
		$$\hat{\mathcal{N}}_{\lambda}\cup \hat{\mathcal{N}}_{\lambda}^{+}=\left\{tu:t>0,~u\in \mathcal{N}_{\lambda}^{+}\cup \mathcal{N}_{\lambda}^{-}\right\},$$
		\item[$(iii)$] there holds
		$$\overline{\hat{\mathcal{N}}_{\lambda_{\ast}}\cup \hat{\mathcal{N}}_{\lambda_{\ast}}^{+}}= \hat{\mathcal{N}}_{\lambda_{\ast}}\cup \hat{\mathcal{N}}_{\lambda_{\ast}}^{+}\cup \left\{tu:t>0,~u\in \mathcal{N}_{\lambda_{\ast}}^{0}\right\}\cup \left\{0\right\},$$
		\item[$(iv)$] the function $t_{\lambda_{\ast}}$ is  continuous and  $P^{-}:S\cap \overline{\hat{\mathcal{N}}_{\lambda_{\ast}}}\rightarrow \mathcal{N}_{\lambda_{\ast}}^{-}\cup \mathcal{N}_{\lambda_{\ast}}^{0}$ defined by $P^{-}(w)=t_{\lambda_{\ast}}(w)w$ is a homeomorphism, where
		\begin{equation}\label{EP1}
			t_{\lambda_{\ast}}(w)=\left\{
			\begin{array}{rcl}
				t_{\lambda_{\ast}}^{-}(w)&~\mbox{if} ~ w\in \hat{\mathcal{N}}_{\lambda_{\ast}},~~~~~~\\
				t_{\lambda_{\ast}}^{0}(w)&~\mbox{otherwise,}~~~~~~~
			\end{array}
			\right.
		\end{equation}
		\item[$(v)$] the function $s_{\lambda_{\ast}}$ is  continuous and $P^{+}:S\rightarrow \mathcal{N}_{\lambda_{\ast}}^{+}\cup \mathcal{N}_{\lambda_{\ast}}^{0}$ defined by $P^{+}(u)=s_{\lambda_{\ast}}(u)u$ is a homeomorphism, where 
		\begin{equation}\label{EP2}
			s_{\lambda_{\ast}}(u)=\left\{
			\begin{array}{rcl}
				t_{\lambda_{\ast}}^{+}(u)&~\mbox{if} ~ u\in \hat{\mathcal{N}}_{\lambda_{\ast}}\cup \hat{\mathcal{N}}_{\lambda_{\ast}}^{+}~~~~~~~~~~~~\\
				t_{\lambda_{\ast}}^{0}(u)&~\mbox{otherwise,}~~~~~~~~~~~~~~~~~~~~
			\end{array}
			\right.
		\end{equation}
		\item[$(vi)$] the set $\mathcal{N}_{\lambda_{\ast}}^{0}\subset \mathcal{N}_{\lambda_{\ast}}$ has empty interior, where  $\mathcal{N}_{\lambda_{\ast}}$ is endowed with the induced topology of the norm on $X$. 
	\end{enumerate}
\end{proposition}

As a fundamental ingredient to show multiplicity of solutions for Problem (\ref{pq})  beyond Nehari's extremal value, we have to prove the continuity and monotonicity of the energy functional constrained on $\mathcal{N}_{\lambda}^+$ and $\mathcal{N}_{\lambda}^-$. To do these, let us define  $J_{\lambda}^{+}:\hat{\mathcal{N}}_{\lambda}\cup \hat{\mathcal{N}}_{\lambda}^{+}\rightarrow \mathbb{R}$ and $J_{\lambda}^{-}:\hat{\mathcal{N}}_{\lambda}\rightarrow \mathbb{R}$  by
\begin{equation}
	\label{2002}
	J^{+}_{\lambda}(u)=\Phi_{\lambda}(t^{+}_{\lambda}(u)u)~\mbox{and}~J_{\lambda}^{-}(u)=\Phi_{\lambda}(t^{-}_{\lambda}(u)u)
\end{equation}
and denote their infimum by 
$$\tilde{J}^{+}_{\lambda}=\inf\left\{J^{+}_{\lambda}(u):u\in \mathcal{N}_{\lambda}^{+}\right\}~\mbox{and}~\tilde{J}^{-}_{\lambda}=\inf\left\{J^{-}_{\lambda}(u):u\in \mathcal{N}_{\lambda}^{-}\right\},$$
respectively.

Unlikely of the non-singular case, the proof of the regularities of the functions $t^{+}_{\lambda}(u)$ and $t^{-}_{\lambda}(u)$ here are  more delicated. However, by inspiring on ideas found in \cite{HSS}, we are able to overcome these obstacles. 
\begin{lemma}\label{L5}
	Let $u\in X_{+}$ and $I\subset \mathbb{R}$ be an open interval such that $t^{\pm}_{\lambda}(u)$ are well defined for all $\lambda \in I$. Then:
	\begin{itemize}
		\item[$a)$] the functions $I\ni \lambda \rightarrow t^{\pm}_{\lambda}
		(u)$ are $C^{\infty}$. Moreover, $I\ni \lambda \rightarrow t^{-}_{\lambda}(u)$ is decreasing while $I\ni \lambda \rightarrow t^{+}_{\lambda}(u)$ is increasing. 
		\item[$b)$] the functions $I\ni \lambda \rightarrow J^{\pm}_{\lambda}(u)$ are $C^{\infty}$ and decreasing.
	\end{itemize}
	In particular, both claims hold true for $I=(0,\lambda_\ast)$ and all $u\in X_{+}$ given.
\end{lemma}
\begin{proof} Let us begin proving $a)$.  To show that $I\ni \lambda \rightarrow t^{\pm}_{\lambda}
	(u)$ are $C^{\infty}$, define the $C^{\infty}$-function $F$ by 
	$$F(\lambda,t,e,f,g)=et-\lambda ft^{-\gamma}- gt^{p}~\mbox{ for}~ (\lambda,t,e,f,g)\in I\times (0,\infty)\times \mathbb{R}^{3},$$
	and set
	$$e_{1}=||u||^{2},~f_{1}=\int_{\mathbb{R}^{N}}a|u|^{1-\gamma}~\mbox{and}~ g_{1}=\int_{\mathbb{R}^{N}}b|u|^{p+1}.$$
	
	For $\lambda^{'}\in I$, we have that
	\begin{align*}
		\frac{\partial F(\lambda^{'},t^{+}_{\lambda^{'}}(u),e_{1},f_{1},g_{1})}{\partial t}=||u||^{2}+\gamma (t^{+}_{\lambda^{'}}(u))^{-\gamma-1}\lambda^{'}\int_{\mathbb{R}^{N}}a(x)|u|^{1-\gamma}dx\\
		-p(t^{+}_{\lambda^{'}}(u))^{p-1}\int_{\mathbb{R}^{N}}b(x)|u|^{p+1}dx>0                                                ,&
	\end{align*}
	because $t^{+}_{\lambda^{'}}(u)u\in \mathcal{N}_{\lambda^{'}}^{+}$. Since
	$$
	F(\lambda^{'},t^{+}_{\lambda^{'}}(u),e_{1},f_{1},g_{1})=0~\mbox{and}~\frac{\partial F}{\partial t}(\lambda^{'},t^{+}_{\lambda^{'}}(u),e_{1},f_{1},g_{1})>0,
	$$
	it follows from the implicit function theorem  that $t^{+}_{\lambda}(u)\in C^{\infty}((\lambda^{'}-\epsilon,\lambda^{'}+\epsilon),\mathbb{R})$ for some $\epsilon>0$ and hence, by the arbitrariness of  $\lambda^{'}$, we conclude that  the function $I\ni \lambda \rightarrow t^{+}_{\lambda}
	(u)$ is $C^{\infty}$. Moreover, since $F(\lambda,t^{+}_{\lambda}(u),e_{1},f_{1},g_{1})=0$ we also have  
	$$\frac{\partial F(\lambda,t^{+}_{\lambda}(u),e_{1},f_{1},g_{1})}{\partial \lambda}+\frac{\partial F(\lambda,t^{+}_{\lambda}(u),e_{1},f_{1},g_{1})}{\partial t}\frac{dt^{+}_{\lambda}(u)}{d\lambda}=0,$$
	that is,
	$$\frac{dt^{+}_{\lambda}(u)}{d\lambda}=\frac{(t^{+}_{\lambda}(u))^{-\gamma}\int_{\mathbb{R}^{N}}a|u|^{1-\gamma}}{||u||^{2}+\gamma (t^{+}_{\lambda}(u))^{-\gamma-1}\lambda\int_{\mathbb{R}^{N}}a(x)|u|^{1-\gamma}dx-p(t^{+}_{\lambda}(u))^{p-1}\int_{\mathbb{R}^{N}}b(x)|u|^{1+p}dx}>0,$$
	where the last inequality is a consequence of  $t^{+}_{\lambda}(u)u\in \mathcal{N}^{+}_{\lambda}$. Therefore, the function $I\ni \lambda \rightarrow t^{+}_{\lambda}(u)$ is increasing. In a similar way, we can prove that $I\ni \lambda \rightarrow t^{-}_{\lambda}(u)$ is $C^{\infty}$ and decreasing.
	
	Now let us prove $b)$. Since $t^{+}_{\lambda}(u)>0$ and 
	\begin{align*}
		J_{\lambda}^{+}(u)=\Phi_{\lambda}(t^{+}_{\lambda}(u)u)=\frac{(t^{+}_{\lambda}(u))^{2}}{2}||u||^{2}-\frac{(t^{+}_{\lambda}(u))^{1-\gamma}\lambda}{1-\gamma}\int_{\mathbb{R}^{N}}a(x)|u|^{1-\gamma}dx\\ -\frac{(t^{+}_{\lambda}(u))^{p+1}}{p+1}\int_{\mathbb{R}^{N}}b(x)|u|^{p+1}dx,
	\end{align*}
	it follows from item $a)$ the $C^{\infty}$-regularity for $J_{\lambda}^{+}(u)$ with respect to $\lambda$. Besides this, we have
	\begin{align*}
		\frac{dJ^{+}_{\lambda}(u)}{d\lambda}=\phi_{\lambda,u}^{'}(t^{+}_{\lambda}(u))\frac{dt^{+}_{\lambda}(u)}{d\lambda}-\frac{(t^{+}_{\lambda}(u))^{1-\gamma}}{1-\gamma}\int_{\mathbb{R}^{N}}a(x)|u|^{1-\gamma}dx\\
		=-\frac{(t^{+}_{\lambda}(u))^{1-\gamma}}{1-\gamma}\int_{\mathbb{R}^{N}}a(x)|u|^{1-\gamma}dx<0,~~~~~~~~~~~~~~~~~~~
	\end{align*}
	where we used the fact that $t^{+}_{\lambda}(u)u\in \mathcal{N}^{+}_{\lambda}$ to obtain the last inequality, that is, $I\ni \lambda \rightarrow J^{+}_{\lambda}(u)$ is decreasing. Similarly, we can prove that $I\ni \lambda \rightarrow J^{-}_{\lambda}(u)$ is a continuous and decreasing function. 
\end{proof} 
\fim
\vspace{0.4cm}

As a consequence of the monotonicity proved above, after some adjusts on the proof of Corollary 2.15  in \cite{SM}, we can prove the below Corollary.
\begin{corollary}\label{C3} Suppose that $u\not\in \hat{\mathcal{N}}_{\lambda_{\ast}}^{+}$. Then
	$$\lim_{\lambda \uparrow \lambda_{\ast}}t^{-}_{\lambda}(u)=t_{\lambda_{\ast}}(u),~\lim_{\lambda \uparrow \lambda_{\ast}}t^{+}_{\lambda}(u)=s_{\lambda_{\ast}}(u)$$
	$$\lim_{\lambda \uparrow \lambda_{\ast}}J^{-}_{\lambda}(u)=\Phi_{\lambda_{\ast}}(t_{\lambda_{\ast}}(u)u),~\lim_{\lambda \uparrow \lambda_{\ast}}J^{+}_{\lambda}(u)=\Phi_{\lambda_{\ast}}(s_{\lambda_{\ast}}(u)u),$$
	where $t_{\lambda_{\ast}}(u)$ and $s_{\lambda_{\ast}}(u)$ are defined at $(\ref{EP1})$ and $(\ref{EP2})$, respectively.
\end{corollary}

\section{Multiplicity of solutions on the interval $0<\lambda<\lambda_{\ast}$}

In this section we show the existence of two solutions for problem \eqref{pq} when $\lambda\in(0,\lambda_\ast)$. Some ideas are motivated by the work of Hirano-Sacon-Shioji \cite{HSS}. Like them, first we show the existence of $u_{\lambda}\in \mathcal{N}^{+}_{\lambda}$ and $w_{\lambda}\in \mathcal{N}^{-}_{\lambda}$ such that
$$\Phi_{\lambda}(u_{\lambda})=\tilde{J}^{+}_{\lambda},~~\Phi_{\lambda}(w_{\lambda})=\tilde{J}^{-}_{\lambda},$$
$$0\leq \int_{\mathbb{R}^{N}}\nabla u_{\lambda} \nabla \psi+V(x)u_{\lambda} \psi dx-\lambda\int_{\mathbb{R}^{N}}a(x)u^{-\gamma}_{\lambda} \psi dx-\int_{\mathbb{R}^{N}}b(x)u^{p}_{\lambda} \psi dx, \forall \psi \in X_{+}$$
and
$$0\leq \int_{\mathbb{R}^{N}}\nabla w_{\lambda} \nabla \psi+V(x)w_{\lambda} \psi dx-\lambda\int_{\mathbb{R}^{N}}a(x)w^{-\gamma}_{\lambda} \psi dx-\int_{\mathbb{R}^{N}}b(x)w^{p}_{\lambda} \psi dx, \forall \psi \in X_{+}.$$
The next step will be to adjust the arguments used to prove the Step 3 of Lemma \ref{L2} to show that the last inequalities are in fact equalities, that is, $u_{\lambda}\in \mathcal{N}^{+}_{\lambda}$ and $w_{\lambda}\in \mathcal{N}^{-}_{\lambda}$
are solutions for problem $(P_{\lambda})$.

To carry out this strategy, let us begin by proving the next Lemma.
\begin{lemma}\label{LL1} Let $\lambda>0$. Then:
	\begin{itemize}
		\item[$a)$] for all $u\in \mathcal{N}^{+}_{\lambda}$, we have that
		\begin{equation}\label{N2}
			||u||^{2}<\frac{\lambda (\gamma+p)}{p-1}\int_{\mathbb{R}^{N}}a(x)|u|^{1-\gamma}dx
		\end{equation}
		holds. In particular $\sup \left\{||u||:u\in \mathcal{N}^{+}_{\lambda} \right\}< \infty$.
		\item[$b)$] for all $w\in \mathcal{N}^{-}_{\lambda}$, we have that
		\begin{equation}\label{N3}
			||w||^{2}<\frac{(\gamma+p)}{(1+\gamma)}\displaystyle \int_{\mathbb{R}^{N}} b|w|^{p+1}
		\end{equation}
		holds and $\sup \left\{||w||:w\in \mathcal{N}^{-}_{\lambda}, \Phi_{\lambda}(w)\leq M \right\}<\infty $ for each $M>0$ given. Moreover 
		$$\inf \left\{||w||:w\in \mathcal{N}^{-}_{\lambda} \right\}>0.$$
	\end{itemize}
	Furthermore, 
	\begin{equation}
		\label{123}
		0>\tilde{J}^{+}_{\lambda}:=\displaystyle\inf_{u\in \mathcal{N}^{+}_{\lambda}}  \Phi_{\lambda}(u)>-\infty~~\mbox{and}~~\tilde{J}^{-}_{\lambda}:=\displaystyle\inf_{w\in \mathcal{N}^{-}_{\lambda}}  \Phi_{\lambda}(w)>-\infty.
	\end{equation}
\end{lemma}
\begin{proof} Item $a)$ is a consequence of $\phi^{''}_{\lambda,u}(1)>0$, H\"older and Sobolev embedding. The inequalities \eqref{N3} of $b)$ and $\inf \left\{||w||:w\in \mathcal{N}^{-}_{\lambda} \right\}>0$ are direct consequences of $\phi^{''}_{\lambda,u}(1)<0$, H\"older and Sobolev embedding. Now fix $M>0$ and $w\in \mathcal{N}_\lambda^-$ such that $\Phi_{\lambda}(w)\leq M$. By using H\"older and Sobolev embeddings, we obtain
	$$ \left(\frac{1}{2}-\frac{1}{p+1}\right)||w||^{2}+\lambda\left(\frac{1}{p+1}-\frac{1}{1-\gamma}\right)C||w||^{1-\gamma}\leq \Phi_{\lambda}(w)\leq M,$$
	where $C$ is a positive constant. Since $0<1-\gamma<2$, we have $\sup \left\{||w||:w\in \mathcal{N}^{-}_{\lambda}, \Phi_{\lambda}(w)\leq M \right\}<\infty $. 
	
	Now, let us prove the two first inequalities in (\ref{123}). First, let $u_{n}\subset \mathcal{N}^{+}_{\lambda}$ such that $\Phi_{\lambda}(u_{n})\rightarrow \tilde{J}_{\lambda}^{+}$. Thus, if follows from the boundedness of $\mathcal{N}^{+}_{\lambda}$ proved in $a)$ that, up to a subsequence, $u_{n}\rightharpoonup u$ in $X$ and hence $-\infty <\Phi_{\lambda}(u)\leq \liminf \Phi_{\lambda}(u_{n})=\tilde{J}_{\lambda}^{+}$. To show the first inequality, we use (\ref{N2}) in the expression of $\Phi_{\lambda}(u)$ to infer that
	\begin{align*}
		\Phi_{\lambda}(u)&=\left(\frac{p-1}{2(p+1)}\right)||u||^{2}-\lambda\left(\frac{\gamma+p}{(p+1)(1-\gamma)}\right) \int_{\mathbb{R}^{N}}a(x)|u|^{1-\gamma}dx \\
		&< \left(\frac{p-1}{2(p+1)}\right)||u||^{2}
		-\left(\frac{(\gamma+p)(p-1)}{(p+1)(1-\gamma)(\gamma+p)}\right)||u||^{2}~~~~~~~\\
		&=-\left(\frac{(1+\gamma)(p-1)}{2(1-\gamma)(p+1)}\right)||u||^{2}<0
	\end{align*}
	holds, that is, $\tilde{J}_{\lambda}^{+}<0$.
	
	In a similar way we can prove that $-\infty <\Phi_{\lambda}(w)\leq \liminf \Phi_{\lambda}(w_{n})=\tilde{J}_{\lambda}^{-}$. This ends the proof.
\end{proof}
\fim
\vspace{0.4cm}

 Now we show that the infimum value is achieved in both Nehari manifolds $ \mathcal{N}^{+}_{\lambda}$ and $\mathcal{N}^{-}_{\lambda}$.

\begin{lemma}\label{LL2} Let $0< \lambda < \lambda_\ast$. Then there exist $u_{\lambda}\in \mathcal{N}^{+}_{\lambda}$ and $w_{\lambda}\in \mathcal{N}^{-}_{\lambda}$ such that $\Phi_{\lambda}(u_{\lambda})=\tilde{J}_{\lambda}^{+}$ and $\Phi_{\lambda}(w_{\lambda})=\tilde{J}_{\lambda}^{-}$.
\end{lemma}
\begin{proof} First, we will show that there exists $u_{\lambda}\in \mathcal{N}^{+}_{\lambda}$ be such that $\Phi_{\lambda}(u_{\lambda})=\tilde{J}_{\lambda}^{+}$. Let $\left\{u_{n}\right\}\subset \mathcal{N}^{+}_{\lambda}$ such that $\Phi_{\lambda}(u_{n})\rightarrow \tilde{J}_{\lambda}^{+}$. So, it follows from Lemma \ref{LL1} a) that, up to a subsequence, $u_{n}\rightharpoonup u_{\lambda}$ in $X$ and $u_{\lambda}\geq 0$. Suppose on the contrary that $u_{\lambda}=0$, then  $0=\Phi_{\lambda}(u_{\lambda})\leq \liminf \Phi_{\lambda}(u_{n})=\tilde{J}_{\lambda}^{+}<0$, which is impossible, that is, $u_{\lambda}\neq 0$ and so $u_{\lambda}\in X_{+}$. 
	
	Let us prove that $u_{\lambda}\in \mathcal{N}^{+}_{\lambda}$. First, we claim that $\left\{u_{n}\right\}$ converges strongly to $u_{\lambda}$ in $X$. On the contrary, we would have that $||u_{\lambda}||<\liminf ||u_{n}||$ and thus
	$$\liminf_{n\to \infty}\phi^{'}_{\lambda,u_{n}}(t^{+}_{\lambda}(u_{\lambda})u_{n})>\phi^{'}_{\lambda,u_{\lambda}}(t^{+}_{\lambda}(u_{\lambda})u_{\lambda})=0,$$
	which implies that $\phi^{'}_{\lambda,u_{n}}(t^{+}_{\lambda}(u)u_{n})>0$ for sufficiently large $n$. It follows from Proposition \ref{prop21} and Lemma \ref{L1} applied to the fiber map $\phi_{\lambda,u_{n}}$ that $1=t^{+}_{\lambda}(u_{n})<t^{+}_{\lambda}(u_{\lambda})$ holds for larger $n$. Therefore, by coming back to the fiber map $\phi_{\lambda,u_{\lambda}}$, we obtain from  Proposition \ref{prop21} again that  $\Phi_{\lambda}(t^{+}_{\lambda}(u_{\lambda})u_{\lambda})<\Phi_{\lambda}(u_{\lambda})$ and consequently
	$$\tilde{J}_{\lambda}\leq J^{+}_{\lambda}(u)=\Phi_{\lambda}(t^{+}_{\lambda}(u)u)<\liminf \Phi_{\lambda}(u_{n})=\tilde{J}^{+}_{\lambda},$$
	which is an absurd, that is, $u_{n}\rightarrow u$ in $X$ and hence
	\begin{equation}
		\label{124}
		\phi^{'}_{\lambda,u_{\lambda}}(1)=\lim_{n \to \infty} \phi^{'}_{\lambda,u_{n}}(1)=0~~\mbox{and}~~\phi^{''}_{\lambda,u_{\lambda}}(1)=\lim_{n \to \infty} \phi^{''}_{\lambda,u_{n}}(1)\geq 0.
	\end{equation}
	Since from Lemma \ref{L1} $b)$ we have that $\mathcal{N}_{\lambda}^{0}=\emptyset$ for $0< \lambda < \lambda_\ast$, we must conclude that $u_{\lambda}\in \mathcal{N}_{\lambda}^{+}$ and $\Phi_{\lambda}(u_{\lambda})=\tilde{J}_{\lambda}^{+}$. 
	
	Next, let us prove that there exists $w_{\lambda}\in \mathcal{N}^{-}_{\lambda}$ for which $\Phi_{\lambda}(w_{\lambda})=\tilde{J}_{\lambda}^{-}$ holds. Let $\left\{w_{n}\right\}\subset \mathcal{N}^{-}_{\lambda}$ be such that $\Phi_{\lambda}(w_{n})\rightarrow \tilde{J}_{\lambda}^{-}$. As above, we have that  $w_{n}\rightharpoonup w_{\lambda}$ in $X$ and $w_{\lambda}\geq 0$. Assume on the contrary that $w_{\lambda}=0$ then, from Lemma \ref{N2} $b)$ we obtain the absurd
	$$0<\inf \left\{||w||:w\in \mathcal{N}^{-}_{\lambda} \right\}\leq \liminf_{n\to \infty} ||w_{n}||^{2}  \leq \liminf_{n\to \infty} \frac{(\gamma+p)}{(1+\gamma)} \displaystyle \int_{\mathbb{R}^{N}} b|w_n|^{p+1}=0$$
	where the last equality follows from the compact embedding $X$ into $L^{p+1}(\mathbb{R}^N)$, hence $w_{\lambda} \neq 0$ and so $w_{\lambda}\in X_{+}$. By repeating the above arguments, we have $\int b|w_{\lambda}|^{p+1}>0$.
	
	We claim that $\left\{w_{n}\right\}$ converges strongly to $w_{\lambda}$ in $X$. Suppose not. Then we may assume that $||w_{n}-w_{\lambda}||\rightarrow \theta>0$ and apply  Brezis-Lieb lemma to infer that
	$$\tilde{J}_{\lambda}^{-}=\Phi_{\lambda}(w_{\lambda})+\frac{\theta^{2}}{2}, ~\phi^{'}_{\lambda,w_{\lambda}}(1)+\theta^{2}=0,~\mbox{and}~\phi^{''}_{\lambda,w_{\lambda}}+\theta^{2}\leq 0$$
	holds. So, we would have  $\phi^{'}_{\lambda,w_{\lambda}}(1)<0$ and $\phi^{''}_{\lambda,w_{\lambda}}(1)<0$. As a consequence of Proposition \ref{prop21} and  Lemma \ref{L1},  there exists a $t_{\lambda}^{-}\in (0,1)$ such that $\phi^{'}_{\lambda,w_{\lambda}}(t_{\lambda}^{-})=0,~\phi^{''}_{\lambda,w_{\lambda}}(t_{\lambda}^{-})<0$ and $t_{\lambda}^{-}w_{\lambda}\in \mathcal{N}^{-}_{\lambda}$. 
	
	By setting $g(t)=\phi_{\lambda,w_{\lambda}}(t)+\frac{\theta^{2}t^{2}}{2}$ for $t>0$ we conclude that   $0<t_{\lambda}^{-}<1,~g'(1)=0$ and $g^{'}(t_{\lambda}^{-})=\theta^{2}t_{\lambda}^{-}>0$, which together with  Proposition \ref{prop21} lead us to conclude that $g$ is increasing on $[t_{\lambda}^{-},1]$. Thus, we have
	$$\tilde{J}_{\lambda}^{-}=\lim \Phi_{\lambda}(w_{n})=g(1)>g(t_{\lambda}^{-})>\phi_{\lambda,w_{\lambda}}(t_{\lambda}^{-})=\Phi_{\lambda}(t_{\lambda}^{-}w_{\lambda})\geq \tilde{J}_{\lambda}^{-},$$
	which is a contradiction, that is $\theta =0$ and $\left\{w_{n}\right\}$ converges strongly to $w_{\lambda}$ in $X$. After this,  we obtain that $w_{\lambda}\in \mathcal{N}_{\lambda}^{-}$ and $\Phi_{\lambda}(w_{\lambda})=\tilde{J}_{\lambda}^{-}$, as done at (\ref{124}). This ends the proof.
\end{proof}
\fim
\vspace{0.4cm}

\begin{lemma}\label{LL3} Let $0<\lambda < \lambda_\ast$. Then there exists $\epsilon_{0}>0$ such that: 
	\begin{itemize}
		\item[a)] $\Phi_{\lambda}(u_{\lambda}+\epsilon \psi )\geq \Phi_{\lambda}(u_{\lambda})$,
		\item[b)] $t_{\lambda}^{-}(w_{\lambda}+\epsilon \psi )\rightarrow 1$ as $\epsilon \downarrow 0$, where $t_{\lambda}^{-}(w_{\lambda}+\epsilon \psi )$ is the unique positive real number, given by Proposition $\ref{prop21}$, satisfying $t_{\lambda}^{-}(w_{\lambda}+\epsilon \psi )(w_{\lambda}+\epsilon\psi )\in \mathcal{N}^{-}_{\lambda}$
	\end{itemize}
	for each $\psi \in X_{+}$ given and for each $0\leq \epsilon \leq \epsilon_{0}$.
\end{lemma}
\begin{proof} Let $\psi$ be a function in $X_{+}$. First,  let us prove $a)$. It follows from (\ref{17}) that
	$$\phi{''}_{\lambda,u_{\lambda}+\epsilon \psi}(1)=||u_{\lambda}+\epsilon \psi||^{2}+\gamma \lambda\int_{\mathbb{R}^{N}}a(x)|u_{\lambda}+\epsilon \psi|^{1-\gamma}dx-p\int_{\mathbb{R}^{N}}b(x)|u_{\lambda}+\epsilon \psi|^{p+1}dx,~\epsilon \geq 0,$$
	which combined with the continuity of $\phi_{\lambda,u_{\lambda}+\epsilon \psi}(1)$ in $\epsilon\geq 0$ and the fact that $\phi^{''}_{\lambda,u_{\lambda}}(1)>0$, because $u_{\lambda}\in \mathcal{N}^{+}_{\lambda}$, implies that there exists an $\epsilon_{0}>0$ such that $\phi{''}_{\lambda,u_{\lambda}+\epsilon \psi}(1)>0$ for all $0\leq \epsilon \leq \epsilon_{0}$.
	
	Fix $0\leq \epsilon \leq \epsilon_{0}$.  Then from $\phi{''}_{\lambda,u_{\lambda}+\epsilon \psi}(1)>0$, we obtain
	$$\Phi_{\lambda}(u_{\lambda}+\epsilon \psi ) = \phi_{\lambda, u_{\lambda}+\epsilon \psi }(1)\geq \phi_{\lambda, u_{\lambda}+\epsilon \psi }(t_{\lambda}^{+}(u_{\lambda}+\epsilon \psi))=
	\Phi_{\lambda}(t_{\lambda}^{+}(u_{\lambda}+\epsilon \psi)( u_{\lambda}+ \epsilon \psi))\geq \Phi_{\lambda}(u_{\lambda})$$
	where the last inequality follows from Lemma \ref{LL2}, because $u_\lambda,t_{\lambda}^{+}(u_{\lambda}+\epsilon \psi)( u_{\lambda}+ \epsilon \psi) \in \mathcal{N}^{+}_{\lambda}$.
	
	Now we prove $b)$. By defining $F:(0,\infty)\times \mathbb{R}^{3}\rightarrow \mathbb{R}$ by
	$F(t,e,f,g)=et-\lambda ft^{-\gamma}- gt^{p}$, we have that 
	$F$ is a $C^{\infty}$ function,
	$$F(1,e_{1},f_{1},g_{1})=\phi^{'}_{\lambda,w_{\lambda}}(1)=0,$$
	because $w_\lambda \in \mathcal{N}_{\lambda}$,
	and 
	$$\frac{dF}{dt}(1,e_{1},f_{1},g_{1})=\phi^{''}_{\lambda,w_{\lambda}}(1)<0,$$
	due to the fact that 
	$w_\lambda \in \mathcal{N}^{-}_{\lambda}$, where
	$$e_{1}=||w_{\lambda}||^{2},~f_{1}=\int_{\mathbb{R}^{N}}a|w_{\lambda}|^{1-\gamma}~\mbox{and}~ g_{1}=\int_{\mathbb{R}^{N}}b|w_{\lambda}|^{p+1}.$$
	
	Therefore, it follows from the implicit function theorem and from $$F(t_{\lambda}^{-}(w_{\lambda}+\epsilon \psi),||w_{\lambda}+\epsilon \psi ||^{2},\int_{\mathbb{R}^{N}}a(x)|w_{\lambda}+\epsilon \psi |^{1-\gamma}dx,\int_{\mathbb{R}^{N}}b(x)|w_{\lambda}+\epsilon \psi |^{p+1}dx)=0,$$
	thanks to Proposition \ref{prop21}, that 
	$$ t(||w_{\lambda}+\epsilon \psi ||^{2},\int_{\mathbb{R}^{N}}a(x)|w_{\lambda}+\epsilon \psi |^{1-\gamma}dx,\int_{\mathbb{R}^{N}}b(x)|w_{\lambda}+\epsilon \psi |^{p+1}dx)=t_{\lambda}^{-}(w_{\lambda}+\epsilon \psi)$$
	for $\epsilon>0$ small enough, where $t:B\rightarrow A$ is a $C^{\infty}$-function where $A$  and $B$ are open neighborhoods of $1$ and $(e_{1},f_{1},g_{1})$, respectively. The  continuity of $t$ implies the claim. This finishes the proof.
\end{proof}
\fim
\vspace{0.4cm}

 Lemma \ref{LL3} implies
\begin{lemma}\label{LL4} Let $0<\lambda < \lambda_\ast$. Then for each $\psi\in X_{+}$ given, there hold $au_{\lambda}^{-\gamma} \psi, aw_{\lambda}^{-\gamma} \psi \in L^{1}(\mathbb{R}^{N})$,
	\begin{equation}
		\label{128}
		\int_{\mathbb{R}^{N}}\nabla u_{\lambda} \nabla \psi +V(x)u_{\lambda} \psi dx-\int_{\mathbb{R}^{N}}(\lambda a(x)u_{\lambda}^{-\delta}\psi dx+b(x)u_{\lambda}^{p}\psi dx)\geq 0
	\end{equation}
	and
	\begin{equation}
		\label{129}
		\int_{\mathbb{R}^{N}}\nabla u_{\lambda} \nabla \psi +V(x)u_{\lambda} \psi dx-\int_{\mathbb{R}^{N}}(\lambda a(x)u_{\lambda}^{-\delta}\psi dx+b(x)u_{\lambda}^{p}\psi dx)\geq 0
	\end{equation}
	
	In particular, $u_{\lambda}, w_{\lambda}>0$ almost everywhere in $\mathbb{R}^{N}$.
\end{lemma}
\begin{proof} Let $\psi\in X_{+}$. First, let us prove the  inequality (\ref{128}). After some manipulations, we obtain from Lemma \ref{LL3} item $a)$, that
	$$\lambda \int_{\mathbb{R}^{N}}\frac{a|u_{\lambda}+\epsilon \psi |^{1-\gamma}-a|u_{\lambda}|^{1-\gamma}}{(1-\gamma)\epsilon}dx\leq \frac{||u_{\lambda}+\epsilon \psi ||^{2}-||u_{\lambda}||^{2}}{ 2\epsilon}-\int_{\mathbb{R}^{N}}\frac{b|u_{\lambda}+\epsilon \psi |^{p+1}-b|u_{\lambda}|^{p+1}}{(p+1)\epsilon}dx$$
	holds 
	for sufficiently small $\epsilon>0$. 
	
	By using similar arguments as in the proof of Lemma \ref{L2}, we obtain from the last inequality that $u_{\lambda}>0$ in $\mathbb{R}^{N}$, $au_{\lambda}^{-\gamma} \psi \in L^{1}(\mathbb{R}^{N})$ and 
	$$\int_{\mathbb{R}^{N}}\nabla u_{\lambda} \nabla \psi +V(x)u_{\lambda} \psi dx-\int_{\mathbb{R}^{N}}(\lambda a(x)u_{\lambda}^{-\delta}\psi dx+b(x)u_{\lambda}^{p}\psi dx)\geq 0.$$

To prove \eqref{129}, we note that 
	$$\Phi_{\lambda}(t_{\lambda}^{-}(w_{\lambda}+\epsilon \psi)( w_{\lambda}+ \epsilon \psi) )\geq 
	\Phi_{\lambda}(w_{\lambda}) = \phi_{\lambda, w_{\lambda} }(1)\geq \phi_{\lambda, w_{\lambda} }(t_{\lambda}^{-}(w_{\lambda}+\epsilon \psi))=\Phi_{\lambda}(t_{\lambda}^{-}(w_{\lambda}+\epsilon \psi)w_{\lambda}),$$
	where the first inequality follows from Lemma \ref{LL2} and the second inequality comes from Proposition \ref{prop21}.
	
	After some manipulations, we obtain  from the above inequality that
	\begin{align*}
		t_{\lambda}^{-}(w_{\lambda}+\epsilon \psi)^{2}\frac{||w_{\lambda}+\epsilon \psi ||^{2}-||w_{\lambda}||^{2}}{ 2\epsilon}-t_{\lambda}^{-}(w_{\lambda}+\epsilon \psi)^{p+1}\int_{\mathbb{R}^{N}}\frac{b|w_{\lambda}+\epsilon \psi |^{p+1}-b|w_{\lambda}|^{p+1}}{(p+1)\epsilon}dx\\
		\geq t_{\lambda}^{-}(w_{\lambda}+\epsilon \psi)^{1-\gamma}\lambda\int_{\mathbb{R}^{N}}\frac{a|w_{\lambda}+\epsilon \psi |^{1-\gamma}-a|w_{\lambda}|^{1-\gamma}}{(1-\gamma)\epsilon}dx
	\end{align*}
	holds for $\epsilon>0$ small enough.
	
	So, by applying Lemma \ref{LL3} item $b)$, we obtain $w_{\lambda}>0$ in $\mathbb{R}^{N}$, $aw_{\lambda}^{-\gamma} \psi \in L^{1}(\mathbb{R}^{N})$ and 
	$$\int_{\mathbb{R}^{N}}\nabla w_{\lambda} \nabla \psi +V(x)w_{\lambda} \psi dx-\int_{\mathbb{R}^{N}}(\lambda a(x)w_{\lambda}^{-\delta}\psi dx+b(x)w_{\lambda}^{p}\psi dx)\geq 0$$
	holds. This completes the proof.
\end{proof}
\fim
\vspace{0.4cm}

\begin{proposition}\label{T1} Let $0<\lambda < \lambda_\ast$. Then $u_\lambda\in \mathcal{N}^{+}_{\lambda}$ and $w_\lambda\in \mathcal{N}^{-}_{\lambda}$ are solutions of Problem $(P_{\lambda})$.
\end{proposition}
\begin{proof} First we will show that $u_{\lambda}$ is a solution for $(P_{\lambda})$. To this end, let $\psi \in X$ and define $\Psi_\epsilon=(u_{\lambda}+\epsilon \psi )^{+}\in X_{+}$ for each $\epsilon>0$ given. Therefore, it follows from Lemma \ref{LL4} that the inequality (\ref{128}) holds true with $\Psi_\epsilon$ in the place of $\psi$.
	
	Now, by adapting the proof of Step 3 of Lemma \ref{L2} with 
	$$||u_{\lambda}||^{2}-\lambda\int_{\mathbb{R}^{N}}a(x)|u_{\lambda}|^{1-\gamma}dx-\int_{\mathbb{R}^{N}}b(x)|u_{\lambda}|^{p+1}dx=0~~(\mbox{because}~u_\lambda \in \mathcal{N}_{\lambda})$$
	in the place of (\ref{118}),
	we are able to show that $u_\lambda$ is a solution for Problem $(P_{\lambda})$. In a similar way,  $w_{\lambda}$ will be a solution for $(P_{\lambda})$ as well.
\end{proof}

\fim
\section{Multiplicity of solutions for $\lambda=\lambda_{\ast}$}

In this section we prove the existence of at least two solutions for Problem $(P_{\lambda_\ast})$ by using the multiplicity result given in Proposition \ref{T1}  for $0<\lambda<\lambda_\ast$ and performing a limit process. The next proposition is a consequence of the monotonicities and regularities of the functions   ${t}_{\lambda}^{+}(u), {t}_{\lambda}^{-}(u)$, ${J}_{\lambda}^{+}$ and $ {J}_{\lambda}^{-}$ given by Lemma \ref{L5}.
\begin{proposition}\label{CC1} There holds:
	\begin{itemize}
		\item[$a)$] the functions $(0,\lambda_{\ast}]\ni \lambda \rightarrow \tilde{J}_{\lambda}^{\pm}$ are decreasing and left-continuous for $\lambda \in (0,\lambda_{\ast})$,
		\item[$b)$] $\displaystyle \lim_{\lambda \uparrow \lambda_{\ast}}\tilde{J}^{\pm}_{\lambda}=\tilde{J}^{\pm}_{\lambda_{\ast}}$.
	\end{itemize}
\end{proposition}

\begin{proposition}\label{T2a} The problem $(P_{\lambda_\ast})$ admits at least two solutions $w_{\lambda_\ast}\in \mathcal{N}_{\lambda_\ast}^{-}$ and $u_{\lambda_\ast}\in \mathcal{N}_{\lambda_\ast}^{+}$.
\end{proposition}
\begin{proof}
	First, let us show that there exists a solution $w_{\lambda_{\ast}}\in \mathcal{N}^{-}_{\lambda_{\ast}}$ for  $(P_{\lambda_{\ast}})$. Let $\left\{\lambda_{n}\right\}\subset (0,\lambda_{\ast})$ be such that $\lambda_{n}\uparrow \lambda_{\ast}$ and $\left\{w_{\lambda_{n}}\right\}\subset \mathcal{N}^{-}_{\lambda_{n}}$ as in Proposition \ref{T1}. Suppose on the contrary that $||w_{\lambda_{n}}||\rightarrow \infty $, hence after applying the H\"older inequality, Sobolev embedding and the fact that $w_{\lambda_{n}}\in \mathcal{N}^{-}_{\lambda_{n}}$, we obtain
	\begin{align*}
		J^{-}_{\lambda_{n}}=\Phi_{\lambda_n}(w_{\lambda_n})
		=\left(\frac{1}{2}-\frac{1}{p+1}\right)||w_{\lambda_{n}}||^{2}+\lambda_{n}\left(\frac{1}{p+1}-\frac{1}{1-\gamma}\right)\int_{\mathbb{R}^{N}}a(x)|w_{\lambda_{n}}|^{1-\gamma}dx~~ &\\
		\geq\left(\frac{1}{2}-\frac{1}{p+1}\right)||w_{\lambda_{n}}||^{2}+C\left(\frac{1}{p+1}-\frac{1}{1-\gamma}\right)||w_{\lambda_{n}}||^{1-\gamma},~~~~~~~~~~~~~~~~
	\end{align*}
	which implies by Proposition \ref{CC1} that $\infty >\lim \tilde{J}^{-}_{\lambda_{n}}\geq \infty$, which is a contradiction. Therefore $\left\{w_{\lambda_{n}}\right\}$ is bounded and we can assume that $w_{\lambda_{n}}\rightharpoonup w_{\lambda_{\ast}}$ in $X$,
	\begin{align*}
		\begin{array}{c}
			w_{\lambda_{n}}\rightarrow w_{\lambda_{\ast}}~\mbox{in}~L^{q}(\mathbb{R}^{N}),\forall ~q\in [2,2^{\ast}),~~~~~~~~~~~~~~~~~~~~\\
			w_{\lambda_{n}}\rightarrow w_{\lambda_{\ast}}~\mbox{a.e.}~\mathbb{R}^{N},~~~~~~~~~~~~~~~~~~~~~~~~~~~~~~~~~~~~~~~~\\
			\mbox{there exist}~ h_{q}\in L^{q}(\mathbb{R}^{N})~\mbox{such that}~|w_{\lambda_{n}}|\leqslant h_{q}~~~~~~
		\end{array}
	\end{align*}
	with $w_{\lambda_{\ast}}\geqslant 0$. 
	
	Thus, once  $w_{\lambda_{n}}$ is a solution for Problem $(P_{\lambda_n})$ it follows that
	\begin{equation}\label{130}
		(w_{\lambda_{\ast}},\psi )-\int_{\mathbb{R}^{N}}b(x)w_{\lambda_{\ast}}^{p}\psi dx\geq \lambda_{\ast}\int_{\mathbb{R}^{N}}a(x)G(x)\psi dx
	\end{equation}
	for all $\psi \in X_+$, where $G$ is understood as $G(x):=w_{\lambda_{\ast}}^{-\gamma}(x)$ if $w_{\lambda_{\ast}}(x)\neq 0$ and $G(x):=\infty$ if $w_{\lambda_{\ast}}(x)=0$. It follows that $0\leq \int_{\mathbb{R}^{N}}a(x)G(x)\psi dx< \infty$, which implies  $w_{\lambda_{\ast}}(x)>0$ in $\mathbb{R}^{N}$ and
	\begin{equation}\label{2000}
		(w_{\lambda_{\ast}},\psi )-\int_{\mathbb{R}^{N}}b(x)w_{\lambda_{\ast}}^{p}\psi dx\geq \lambda_{\ast}\int_{\mathbb{R}^{N}}a(x)w_{\lambda_{\ast}}^{-\gamma}\psi dx, \forall \psi \in X_{+}.
	\end{equation}
	
	Moreover, it follows from Lemma \ref{lemma2} and Fatou's lemma that
	\begin{align*}
		\limsup_{n \to \infty}(w_{\lambda_{n}},w_{\lambda_{n}}-w_{\lambda_{\ast}})
		\leq  \limsup_{n \to \infty} \lambda_{n}\int_{\mathbb{R}^{N}}a(x)w_{\lambda_{n}}^{1-\gamma}+\limsup_{n \to \infty} \left(-\lambda_{n}\int_{\mathbb{R}^{N}}a(x)w_{\lambda_{n}}^{-\gamma}w_{\lambda_{\ast}}\right)~~~~~~~~~~~~~~~~~~~~~ \\
		=\lambda_{\ast}\int_{\mathbb{R}^{N}}a(x)w_{\lambda_{\ast}}^{1-\gamma}-\liminf\int_{\mathbb{R}^{N}}\lambda_{n} a(x)w_{\lambda_{n}}^{-\gamma}w_{\lambda_{\ast}}~~~~~~~~~~~~~~~~~~~~~~~~~~~~~~~~~~~~~~~\\
		\leq  \lambda_{\ast}\int_{\mathbb{R}^{N}}a(x)w_{\lambda_{\ast}}^{1-\gamma}-\int_{\mathbb{R}^{N}}\liminf \lambda_{n} a(x)w_{\lambda_{n}}^{-\gamma}w_{\lambda_{\ast}}~~~~~~~~~~~~~~~~~~~~~~~~~~~~~~~~~~~~~~~\\
		=\lambda_{\ast}\int_{\mathbb{R}^{N}}a(x)w_{\lambda_{\ast}}^{1-\gamma}-\lambda_{\ast}\int_{\mathbb{R}^{N}}a(x)w_{\lambda_{\ast}}^{1-\gamma}=0~~~~~~~~~~~~~~~~~~~~~~~~~~~~~~~~~~~~~~~~~~~~~~
	\end{align*}
 that is,
	$$\limsup ||w_{\lambda_{n}}-w_{\lambda_{\ast}}||^{2}\leq \limsup (w_{\lambda_{n}},w_{\lambda_{n}}-w_{\lambda_{\ast}})+\limsup -(w_{\lambda_{\ast}},w_{\lambda_{n}}-w_{\lambda_{\ast}})\leq 0,$$
	which  implies that  $w_{\lambda_{n}}\to w_{\lambda_{\ast}}$ in $X$.
	
	As a consequence of this, we have that 
	$$\phi^{'}_{\lambda_{\ast},w_{\lambda_{\ast}}}(1)=\lim \phi^{'}_{\lambda_{n},w_{\lambda_{n}}}(1)=0~\mbox{and}~\phi^{''}_{\lambda_{\ast},w_{\lambda_{\ast}}}(1)=\lim \phi^{''}_{\lambda_{n},w_{\lambda_{n}}}(1)\leq 0$$
	which implies, by the first equality, that  $w_{\lambda_{\ast}}\in \mathcal{N}_{\lambda_{\ast}}$. We also have from  Lemma \ref{LL1} $b)$, that
	$$0<(1+\gamma)||w_{\lambda_{\ast}}||=(1+\gamma)\displaystyle\lim_{n\to \infty}||w_{\lambda_{n}}||\leq (\gamma+p)\displaystyle\lim_{n\to \infty} \int_{\mathbb{R}^{N}}b(x)w_{\lambda_{n}}^{p+1}dx= (\gamma+p)\int_{\mathbb{R}^{N}}b(x)w_{\lambda_{\ast}}^{p+1}dx,$$
	that is, $\int_{\mathbb{R}^{N}}b(x)w_{\lambda_{\ast}}^{p+1}dx>0$ and hence $w_{\lambda_{\ast}} \in \mathcal{N}^{-}_{\lambda_{\ast}}\cup \mathcal{N}^{0}_{\lambda_{\ast}}$.
	
	By using that $w_{\lambda_{\ast}}\in \mathcal{N}_{\lambda_{\ast}}$, that is, 
	$$|| w_{\lambda_{\ast}}||^{2}-\lambda_{\ast}\int_{\mathbb{R}^{N}}a(x)|w_{\lambda_{\ast}}|^{1-\gamma}dx-\int_{\mathbb{R}^{N}}b(x)|w_{\lambda_{\ast}}|^{p+1}dx=0$$
	holds,
	taking $\Psi_{\epsilon}=(w_{\lambda_{\ast}}+\epsilon \psi )^{+}\in X_{+}$, for $\psi \in X, \epsilon>0$ given, as a test function in (\ref{2000}) and following similar arguments as done in the proof of the Proposition \ref{T1}, we are able to conclude that 
	$w_{\lambda_{\ast}}$ is a solution of $(P_{\lambda_{\ast}})$. Moreover, 
	$w_{\lambda_{\ast}} \in \mathcal{N}_{\lambda_{\ast}}^{-}$ due to Corollary \ref{C1}. Finally, it follows from the strong convergence, Proposition \ref{T1}, Proposition \ref{CC1} and Proposition \ref{r1} $(iv),(v),(vi)$ that 
	\begin{equation}
		\label{140}
		\Phi_{\lambda_{\ast}}(w_{\lambda_{\ast}})= \lim \Phi_{\lambda_{n}}(w_{n})=\lim \tilde{J}^{-}_{\lambda_{n}}= \tilde{J}^{-}_{\lambda_{\ast}}=\inf \left\{{\Phi}_{\lambda_{\ast}}(t_{\lambda_{\ast}}(w)w):w\in \mathcal{N}^{-}_{\lambda_{\ast}}\cup \mathcal{N}^{0}_{\lambda_{\ast}}\right\}
	\end{equation}
	holds, that is, $w_{\lambda_{\ast}} \in \mathcal{N}_{\lambda_{\ast}}^{-}$ is a global minimum of  $\Phi_{\lambda_{\ast}}$ constrained to $\mathcal{N}^{-}_{\lambda_{\ast}}\cup \mathcal{N}^{0}_{\lambda_{\ast}}$.
	\smallskip
	
	In order to show the existence of a second solution for Problem $(P_{\lambda_{\ast}})$, we proceed in a similar way, that is, pick a $\left\{\lambda_{n}\right\}\subset (0,\lambda_{\ast})$ such that $\lambda_{n}\uparrow \lambda_{\ast}$ and $\left\{u_{\lambda_{n}}\right\}\subset \mathcal{N}^{+}_{\lambda_{n}}$ as given by Proposition \ref{T1}. After some manipulations, we obtain that 
	$u_{\lambda_{n}}\to u_{\lambda_{\ast}}$ in $X$ for some  $0<u_{\lambda_{\ast}} \in  \mathcal{N}^{-}_{\lambda_{\ast}}\cup \mathcal{N}^{0}_{\lambda_{\ast}}$, which is a solution for Problem $(P_{\lambda_{{\ast}}})$. 
	
	Besides this, if  $\int_{\mathbb{R}^{N}}b(x)u_{\lambda_{\ast}}^{p+1}dx> 0$ and $\phi^{''}_{\lambda_{\ast},u_{\lambda_{\ast}}}(1)= 0$, then $u_{\lambda_{\ast}}$ would be a solution for the problem $(P_{\lambda_{{\ast}}})$ in $\mathcal{N}^{0}_{\lambda_{\ast}}$, but this is impossible by Corollary \ref{C1}. So we have $\phi^{''}_{\lambda_{\ast},u_{\lambda_{\ast}}}(1)> 0$ in this case. On the other side, if  $\int_{\mathbb{R}^{N}}b(x)u_{\lambda_{\ast}}^{p+1}dx\leq 0$, then we have
	$$\phi^{''}_{\lambda_{\ast},u_{\lambda_{\ast}}}(1)=||u_{\lambda_{\ast}}||^{2}+\gamma \lambda_{\ast}\int_{\mathbb{R}^{N}}a(x)u_{\lambda_{\ast}}^{1-\gamma}dx-p\int_{\mathbb{R}^{N}}b(x)u_{\lambda_{\ast}}^{p+1}dx>0.$$
	So, in both cases, we have $\phi^{''}_{\lambda_{\ast},u_{\lambda_{\ast}}}(1)> 0$ which implies that $u_{\lambda_{\ast}}\in \mathcal{N}^{+}_{\lambda_{\ast}}$. We also have that  $u_{\lambda_{\ast}} \in \mathcal{N}_{\lambda_{\ast}}^{-}$ is a global minimum of  $\Phi_{\lambda_{\ast}}$ constrained to $\mathcal{N}^{+}_{\lambda_{\ast}}\cup \mathcal{N}^{0}_{\lambda_{\ast}}$ as well. This ends the proof.
\end{proof}

\fim
\vspace{0.4cm}

Before proving the multiplicity of solutions for Problem $(P_{\lambda})$ when $\lambda > \lambda_\ast$, let us gather further information on the sets
\begin{equation}\label{2001}
	S^{-}_{\lambda_\ast}=\left\{w\in \mathcal{N}_{\lambda_\ast}^{-}:J^{-}_{\lambda_\ast}(w)=\tilde{J}_{\lambda_\ast}^{-}\right\}~~\mbox{and}~~S^{+}_{\lambda_\ast}=\left\{u\in \mathcal{N}_{\lambda_\ast}^{+}:J^{+}_{\lambda_\ast}(u)=\tilde{J}_{\lambda_\ast}^{+}\right\}.
\end{equation}
\begin{corollary}
	\label{139}
	We have  that: 
	\begin{enumerate}
		\item [$a)$] $S^{-}_{\lambda_\ast}$ and $S^{+}_{\lambda_\ast}$ are non-empties, 
		\item [$b)$]  there exist $c_{\lambda_\ast}, C_{\lambda_\ast}>0$ such that $c_{\lambda_\ast} \leq \Vert u \Vert,  \Vert w \Vert \leq C_{\lambda_\ast}$  for all $u \in S^{+}_{\lambda_\ast}$ and $w \in S^{-}_{\lambda_\ast}$, 
		\item [$c)$] if $u\in 	S^{-}_{\lambda_\ast}\cup 	S^{+}_{\lambda_\ast}$, then $u$ is a solution for Problem $(P_{\lambda_{\ast}})$.
	\end{enumerate}
\end{corollary}
\begin{proof}
	The item $a)$ follows immediately from  $(\ref{140})$, while $b)$
	is a consequence of Lemma $\ref{LL1}$. Finally, the proof of the item $c)$ is similar to that of Proposition \ref{T2a}.
\end{proof}

\fim

\section{Multiplicity of solutions for $\lambda>\lambda_{\ast}$}

In this section we show the existence of solutions for problem $(P_{\lambda})$ when $\lambda$ is greater than $\lambda_{\ast}$ but close to it. The idea is to minimize the energy functional $\Phi_{\lambda}$ over subsets of $\mathcal{N}_{\lambda}^{+}$ and $\mathcal{N}_{\lambda}^{-}$, which are projections of subsets of $\mathcal{N}_{\lambda_\ast}^{+}$ and $\mathcal{N}_{\lambda_\ast}^{-}$ that have positive distances to $\mathcal{N}_{\lambda_\ast}^{0}$. To do this, we do a finer analysis on these sets and we obtain new estimates that are new even in the non-singular case as in \cite{SM}.

\begin{proposition}\label{PPP1} Let $c<C$. Assume that $\lambda_{n}\downarrow \lambda_{\ast}$. 
	\begin{itemize}
		\item[a)] suppose that $w_{n}\in \mathcal{N}^{-}_{\lambda_{\ast}}$ satisfies $c\leq ||w_{n}||\leq C$. If $(t^{-}_{\lambda_{n}}(w_{n}))^{2}\phi^{''}_{\lambda_{n},w_n}(t^{-}_{\lambda_{n}}(w_{n}))\rightarrow 0$, 
		then $d(w_{n},\mathcal{N}^{0}_{\lambda_{\ast}})\rightarrow 0$ as $n\rightarrow \infty$,
		\item[b)] suppose that $u_{n}\in \mathcal{N}^{+}_{\lambda_{\ast}}$ satisfies $c\leq ||u_{n}||\leq C$. If $(t^{+}_{\lambda_{n}}(u_{n}))^{2}\phi^{''}_{\lambda_{n},u_n}(t^{+}_{\lambda_{n}}(u_{n}))\rightarrow 0$, then $d(u_{n},\mathcal{N}^{0}_{\lambda_{\ast}})\rightarrow 0$ as $n\rightarrow \infty$.
	\end{itemize}
	
\end{proposition}
\begin{proof} We prove only $a)$ since the proof of $b)$ follows the same strategy. It follows from Lemma \ref{LL1} $b)$ that there exists a positive constant $c$ such that $\int_{\mathbb{R}^{N}} b|w_{n}|^{p+1}\geq c$. We claim that the same holds for $\int_{\mathbb{R}^{N}} a|w_{n}|^{1-\gamma}$. To prove this, let us first prove that $t^{-}_{\lambda_{n}}(w_{n})\rightarrow \theta \in (0,\infty)$.
	
	Now, by applying Proposition \ref{prop21}, there exist   $s_{n}:=t^{+}_{\lambda_{n}}(w_{n})<t^{-}_{\lambda_{n}}(w_{n}):=t_{n}$
	such that
	$$
	\left\{
	\begin{array}{l}
	t_{n}^{2}||w_{n}||^{2}-t_{n}^{1-\gamma}\lambda_{n}\int_{\mathbb{R}^{N}} a|w_{n}|^{1-\gamma}-t_{n}^{p+1} \int_{\mathbb{R}^{N}} b|w_{n}|^{p+1}=0,\\
	t_{n}^{2}||w_{n}||^{2}+t_{n}^{1-\gamma}\lambda_{n}\gamma\int_{\mathbb{R}^{N}} a|w_{n}|^{1-\gamma}-t_{n}^{p+1}p \int_{\mathbb{R}^{N}} b|w_{n}|^{p+1}=o(1),\\
	s_{n}^{2}||w_{n}||^{2}-s_{n}^{1-\gamma}\lambda_{n}\int_{\mathbb{R}^{N}} a|w_{n}|^{1-\gamma}-s_{n}^{p+1} \int_{\mathbb{R}^{N}} b|w_{n}|^{p+1}=0,
	\end{array}
	\right.
	$$
	where the second line is a consequence of the assumption $(t^{-}_{\lambda_{n}}(w_{n}))^{2}\phi^{''}_{\lambda_{n},w_n}(t^{-}_{\lambda_{n}}(w_{n}))\rightarrow 0$.
	
	So, by solving the system formed by the first and third equation of the above system treating the integrals as unknown, and substituting them into the  second equation, we obtain  
	\begin{equation}\label{131}
		||w_{n}||^{2}t_{n}^{2}\left[\frac{(1+\gamma)\left(\frac{s_{n}}{t_{n}}\right)^{p+\gamma}+(p-1)-(\gamma+p)\left(\frac{s_{n}}{t_{n}}\right)^{1+\gamma}}{\left(\frac{s_{n}}{t_{n}}\right)^{p+\gamma}-1}\right]=o(1),~n\rightarrow \infty .
	\end{equation}
	
	Besides this, it follows from  $C\geq ||w_{n}||\geq c$,  Lemma \ref{LL1}, the first and third equations of system above and $s_{n}<t_{n}$ that there exists positive constants $\tilde{c},\tilde{C},\theta, \alpha$ such that $t_{n},s_{n}\in [\tilde{c},\tilde{C}],~t_{n}\rightarrow \theta,~s_{n}\rightarrow \alpha$  and   $||t_{n}w_{n}||\geq \tilde{c}$. By using these informations and taking limit on (\ref{131}),  we conclude that ${s_{n}}/{t_{n}}\rightarrow 1$ and $\theta=\alpha$, because  $t=1$ is the only zero of the function
	$$g(t)=(1+\gamma)t^{p+\gamma}+(p-1)-(\gamma+p)t^{1+\gamma}.$$
	
	Once $s_nw_n \in \mathcal{N}^{+}_{\lambda_{n}}$, we obtain from  Lemma \ref{LL1} $a)$  that $\int a|w_{n}|^{1-\gamma}\geq c$. Follows that
	$$
	\left\{
	\begin{array}{c}
	||\theta w_{n}||^{2}-\lambda_{\ast}\int_{\mathbb{R}^{N}} a|\theta w_{n}|^{1-\gamma}- \int_{\mathbb{R}^{N}} b|\theta w_{n}|^{p+1}=o(1),\\
	||\theta w_{n}||^{2}+\gamma\lambda_{\ast}\int_{\mathbb{R}^{N}} a|\theta w_{n}|^{1-\gamma}-p \int_{\mathbb{R}^{N}} b|\theta w_{n}|^{p+1}=o(1)\\
	\end{array}
	\right.
	$$
	and infer that
	$$\frac{p-1}{\gamma+p}\frac{||\theta w_{n}||^{2}}{\int_{\mathbb{R}^{N}} a|\theta w_{n}|^{1-\gamma}}=\lambda_{\ast}+o(1),~n\rightarrow \infty ,$$
	and
	$$\frac{1+\gamma}{\gamma+p}\frac{||\theta w_{n}||^{2}}{\int_{\mathbb{R}^{N}} b|\theta w_{n}|^{p+1}}=1+o(1),~n\rightarrow \infty.$$
	
	Therefore, it follows from  (\ref{E2}) and by $0$-homogeneity  that
	$$\lambda(w_{n})=\lambda(\theta w_{n})=\left(1+o(1)\right)^{\frac{1+\gamma}{p-1}}\left(\lambda_{\ast}+o(1)\right)\rightarrow \lambda_{\ast},~n\rightarrow \infty,$$
	and  $w_{n}$ is a bounded minimizing sequence for $\lambda_{\ast}$. Moreover, by following similar arguments as done  in the proof  of  Lemma \ref{L4}, we obtain, up to a subsequence, that $w_{n}\rightarrow w\in \mathcal{N}^{0}_{\lambda_{\ast}}$ and consequently $d(w_{n},\mathcal{N}^{0}_{\lambda_{\ast}})\rightarrow 0$ as $n\rightarrow \infty$. This ends the proof.
\end{proof}

\fim
\vspace{0.4cm}

Define
$$\mathcal{N}_{\lambda_{\ast},d,C}^{-}=\left\{w\in \mathcal{N}_{\lambda_{\ast}}^{-}:d(w,\mathcal{N}_{\lambda_{\ast}}^{0})>d,||w||\leq C \right\},$$
and
$$\mathcal{N}_{\lambda_{\ast},d,c}^{+}=\left\{u\in \mathcal{N}_{\lambda_{\ast}}^{+}:d(u,\mathcal{N}_{\lambda_{\ast}}^{0})>d,c\leq ||u|| \right\},$$
for $c,C,d>0$ given. As  an immediately consequence of Proposition \ref{PPP1}, we have.
\begin{corollary}\label{CCC1} Fix $c,C,d>0$. Then there exist $\epsilon>0$ satisfying:  
	\begin{itemize}
		\item[a)] there exists $\delta<0$ such that $(t^{-}_{\lambda}(w))^{2}\phi^{''}_{\lambda,w}(t^{-}_{\lambda}(w))<\delta$ for all $\lambda \in (\lambda_{\ast},\lambda_{\ast}+\epsilon)$ and  $w\in \mathcal{N}^{-}_{\lambda_{\ast},d,C}$. In particular, we have that $t_{\lambda}^-(w)w \in \mathcal{N}_{\lambda}^{-}$ and $w\in \hat{\mathcal{N}}_{\lambda}$ for all $\lambda \in (\lambda_{\ast},\lambda_{\ast}+\epsilon)$,
		\item[b)] there exists $\delta>0$ such that  $(t^{+}_{\lambda}(u))^{2}\phi^{''}_{\lambda}(t^{+}_{\lambda}(u))>\delta$ for all $\lambda \in (\lambda_{\ast},\lambda_{\ast}+\epsilon)$ and $u\in \mathcal{N}^{+}_{\lambda_{\ast},d,c}$. In particular, we have that $t_{\lambda}^+(u)u \in \mathcal{N}_{\lambda}^{+}$ and $u\in \hat{\mathcal{N}}_{\lambda}\cup \hat{\mathcal{N}}_{\lambda}^{+}$ for all $\lambda \in (\lambda_{\ast},\lambda_{\ast}+\epsilon)$.
	\end{itemize}
\end{corollary}

To do a good choice of the parameter $d>0$ in the last corollary, we prove the next result, where  the sets $S^{-}_{\lambda_{\ast}}$ and $S^{+}_{\lambda_{\ast}}$ were defined at (\ref{2001}).
\begin{proposition}\label{PPP2} 
	There holds:
	\begin{itemize}
		\item[a)]$d(S^{-}_{\lambda_{\ast}},\mathcal{N}_{\lambda_{\ast}}^{0})>0,$
		\item[b)]$d(S^{+}_{\lambda_{\ast}},\mathcal{N}_{\lambda_{\ast}}^{0})>0.$
		
	\end{itemize}
	
\end{proposition}
\begin{proof} We just prove $a)$ because the proof of $b)$ follows similar arguments.  Assume by contradiction that $d(S^{-}_{\lambda_{\ast}},\mathcal{N}_{\lambda_{\ast}}^{0})=0$. Then, there exist $w_{n}\in S^{-}_{\lambda_{\ast}}$ and  $v_{n}\in \mathcal{N}_{\lambda_{\ast}}^{0}$ such that $\rVert w_{n}-v_{n}\rVert \rightarrow 0$ as $n\rightarrow \infty $ and
	$$
	(w_{n},\psi)=\lambda_{\ast}\int a w_{n}^{-\gamma}\psi+\int bw_{n}^{p}\psi,~\forall \psi \in X, ~\forall n\in \mathbb{N}
	$$
	holds, where this equality is a consequence of $w_{n}$ being a solution for Problem $(P_{\lambda_{\ast}})$ as claimed in Corollary \ref{139}. Since $\mathcal{N}_{\lambda_{\ast}}^{0}$ is a compact set, see Lemma \ref{L4},  we may assume  that $v_{n}\rightarrow v\in \mathcal{N}_{\lambda_{\ast}}^{0}$ and hence $w_{n}\rightarrow v$ as well. From  Fatous' Lemma  we conclude  that 
	$$
	(v,\psi)\geq \lambda_{\ast}\int a v^{-\gamma}\psi+\int bv^{p}\psi,~\forall \psi \in X_+,
	$$
	that is, we arrived in the same situation as in (\ref{130}) with  $v\in \mathcal{N}_{\lambda_{\ast}}^{0}$. So, by following the same arguments as done after (\ref{130}), we are able to show 
	that $v\in  \mathcal{N}_{\lambda_{\ast}}^{0}$ is a solution for Problem  $(P_{\lambda_{\ast}})$, but this is impossible by Corollary \ref{C1}, which ends the proof.
\end{proof}
\fim
\vspace{0.4cm}

After Corollaries \ref{139}, \ref{CCC1} and Proposition \ref{PPP2}, we are in position to introduce  
\begin{equation}
	\label{defu}
	\tilde{J}^{-}_{\lambda,d^{-},C}\equiv\inf \left\{J^{-}_{\lambda}(w):w\in \mathcal{N}_{\lambda_{\ast},d^{-},C}^{-} \right\} ~~\mbox{and}~~\tilde{J}^{+}_{\lambda,d^{+},c}\equiv\inf \left\{J^{+}_{\lambda}(w):w\in \mathcal{N}_{\lambda_{\ast},d^{+},c}^{+} \right\}
\end{equation}
for each $0<c<c_{\lambda_\ast}$, $C>C_{\lambda_\ast}$ (see Corollary \ref{139} for both) $\lambda_\ast<\lambda < \lambda_\ast +\epsilon$ (see  Corollary \ref{CCC1}) and $0< d^{\pm} < d(S^{\pm}_{\lambda_{\ast}},\mathcal{N}_{\lambda_{\ast}}^{0})$  (see Proposition \ref{PPP2}) which implies that  $S_{\lambda_{{\ast}}}^{-}\subset \mathcal{N}_{\lambda_{\ast},d^{-},C}^{-}$ and $S_{\lambda_{{\ast}}}^{+}\subset \mathcal{N}_{\lambda_{\ast},d^{+},c}^{+}$. The proofs of the next propositions are similar to that of Propositions $4.5$, $4.6, 4.7$  in \cite{SM}. 
\begin{proposition}\label{PPP3} The $\lambda$-functions $\tilde{J}^{-}_{\lambda,d^{-},C}$ and $\tilde{J}^{+}_{\lambda,d^{+},C}$ are decreasing and there holds:
	\begin{itemize}
		\item[a)] $\displaystyle \lim_{\lambda \downarrow \lambda_{\ast}}\tilde{J}^{-}_{\lambda,d^{-},C}=\tilde{J}^{-}_{\lambda_{{\ast}}},$
		\item[b)] $\displaystyle \lim_{\lambda \downarrow \lambda_{\ast}}\tilde{J}^{+}_{\lambda,d^{+},c}=\tilde{J}^{+}_{\lambda_{{\ast}}}.$
	\end{itemize}
\end{proposition}
\begin{proposition}\label{PPP4}  There exists $\epsilon^{-}>0$ such that  ${J}^{-}_{\lambda}$ constrained to $\mathcal{N}_{\lambda_{\ast},d^{-},C}^{-}$ has a minimizer $\tilde{w}_{\lambda}\in \mathcal{N}_{\lambda_{\ast},d^{-},C}^{-}$ for all $\lambda \in (\lambda_{\ast},\lambda_{\ast}+\epsilon^{-})$ given.
\end{proposition}
\begin{proposition}\label{PPP5} There exists $\epsilon^{+}>0$ such that  ${J}^{+}_{\lambda}$ constrained to $\mathcal{N}_{\lambda_{\ast},d^{+},c}^{+}$ has a minimizer $\tilde{u}_{\lambda}\in \mathcal{N}_{\lambda_{\ast},d^{+},c}^{+}$ for all $\lambda \in (\lambda_{\ast},\lambda_{\ast}+\epsilon^{+})$ given.	
\end{proposition}

Unlike the non-singular case, local or global minimizers for the energy functional constrained to Nehari sets, are not necessarily solutions for Problem $(P_{\lambda})$. In the next Proposition we will establish that this claim is true under our assumptions. The main point in order to prove that the minima found in Propositions \ref{PPP4}, \ref{PPP5} are solutions of $(P_{\lambda})$ is to prove that $\tilde{w}_{\lambda}$ and $\tilde{u}_{\lambda}$ are interior points of $\mathcal{N}_{\lambda_{\ast},d^{-},C}^{-}$ and $\mathcal{N}_{\lambda_{\ast},d^{+},c}^{+}$ respectively.
\begin{proposition}\label{T2}
	There exists $\epsilon>0$ such that the problem $(P_{\lambda})$ admits at least two solutions $w_{\lambda}\in \mathcal{N}_{\lambda}^{-}$ and $u_{\lambda}\in \mathcal{N}_{\lambda}^{+}$ for each $ \lambda \in (\lambda_{{\ast}},\lambda_{{\ast}}+\epsilon)$.
\end{proposition}
\begin{proof} First, let us take advantage of the existence of the minimizer $\tilde{w}_{\lambda}\in  \mathcal{N}_{\lambda_{\ast},d^{-},C}^{-}     $ to build a solution for Problem $(P_{\lambda})$ in $\mathcal{N}_{\lambda}^{-}$. Let us do this by reminding that the definitions given at (\ref{defu}) and (\ref{2002}) implies that we can consider  $w_{\lambda}:=t^{-}_{\lambda}(\tilde{w}_{\lambda})\tilde{w}_{\lambda}\in \mathcal{N}_{\lambda}^{-}.$ Below, let us prove that $w_{\lambda}$ is a solution for Problem $(P_{\lambda})$ if $\lambda>\lambda_\ast$ varies in an appropriate range. To this end, firstly we prove that $\tilde{w}_{\lambda}$ is a interior point of $\mathcal{N}_{\lambda_{\ast},d^{-},C}^{-}$ for $\lambda$ close $\lambda_{{\ast}}$, which is equivalently to prove \\
	
	\noindent{\bf Claim:} there exists an $\epsilon_1>0 $ such that 
	\begin{equation}\label{143}
		||\tilde{w}_{\lambda}||<C,~\forall ~\lambda \in (\lambda_{\ast},\lambda_{\ast}+\epsilon_{1}),
	\end{equation}
	where $C>C_{\lambda_{{\ast}}}$ and $C_{\lambda_{{\ast}}}>0$ is given by Corollary \ref{139}.
	
	Indeed, let $\lambda_{n}\downarrow \lambda_{{\ast}}$ and denote $\tilde{w}_{\lambda_{n}}=\tilde{w}_{n}$. Due to the boundedness  of $\mathcal{N}_{\lambda_{\ast},d^{-},C}^{-}$, we may assume that $\tilde{w}_{\lambda_{n}}\rightharpoonup \tilde{w}$ in $X$. In fact, we have that $\tilde{w}_{n}\to \tilde{w}$ in $X$, otherwise we would have
	$||\tilde{w}||<\liminf||\tilde{w}_{n}||$
	which implies
	$$0=\phi^{'}_{\lambda_{\ast},\tilde{w}}(t_{\lambda_{{\ast}}}(\tilde{w}))<\liminf \phi^{'}_{\lambda_{n},\tilde{w}_{n}}(t_{\lambda_{\ast}}(\tilde{w})),$$
	where $t_{\lambda_{{\ast}}}$ is given by Proposition \ref{r1} $(iv)$. It follows that there exists $k$ such that  $\phi^{'}_{\lambda_{n},\tilde{w}_{n}}(t_{\lambda_{\ast}}(\tilde{w}))>0$ for $n\geq k$, that is,  $t^{+}_{n}(\tilde{w}_{n})<t_{\lambda_{\ast}}(\tilde{w})<t^{-}_{n}(\tilde{w}_{n})$ by Proposition \ref{prop21}. Therefore
	
	$$\Vert t_{\lambda_{\ast}}(\tilde{w})\tilde{w} \Vert^2 <\liminf_{{n\to \infty}}\Vert t_{\lambda_{n}}(\tilde{w_n})\tilde{w_n} \Vert^2,$$
	 which lead us to
	\begin{align}
		\label{1361}
		\Phi_{\lambda_{\ast} }(t_{\lambda
			_{\ast}}(\tilde{w})\tilde{w})<\displaystyle \liminf_{\lambda_{n}\downarrow \lambda_{\ast}}\Phi_{\lambda_{n} }(t_{\lambda
			_{\ast}}(\tilde{w})\tilde{w_{n}})
		\leq \displaystyle \liminf_{\lambda_{n}\downarrow \lambda_{\ast}}\Phi_{\lambda_{n} }(t^{-}_{\lambda
			_{n}}(\tilde{w}_{n})\tilde{w}_{n})=\hat{J}^{-}_{\lambda_{\ast}},
	\end{align}
	where the Proposition \ref{PPP3} $a)$ was used to get the last equality. Moreover, it follows from Proposition \ref{CC1} $b)$,  Proposition \ref{T1} and Corollary \ref{C3} that
	\begin{align*}
		\hat{J}_{\lambda_{\ast}}^{-}=\displaystyle \lim_{\lambda_{n}^{'}\uparrow \lambda_{\ast}}\hat{J}_{\lambda_{n}^{'}}^{-}\leq \displaystyle \lim_{\lambda_{n}^{'}\uparrow \lambda_{\ast}}\Phi_{\lambda_{n}^{'} }(t^{-}_{\lambda
			_{n}^{'}}(\tilde{w})\tilde{w})=\Phi_{\lambda_{\ast} }(t_{\lambda
			_{\ast}}(\tilde{w})\tilde{w})
	\end{align*}
	holds for any ${\lambda_{n}^{'}\uparrow \lambda_{\ast}}$. By combining the last inequality with (\ref{1361}) we get a contradiction and hence $\tilde{w}_{n}\to \tilde{w}$ in $X$.

	As a consequence of this strong convergence and  Lemma \ref{LL1} $b)$, we obtain $\int b|\tilde{w}|^{p+1}>0$ and $\phi^{'}_{\lambda_{\ast},\tilde{w}}(1)=0$ and  $\phi^{''}_{\lambda_{\ast},\tilde{w}}(1)\leq 0$, which means by Proposition \ref{prop21} that  
	$\tilde{w}\in \mathcal{N}_{\lambda_{\ast}}^{-} \cup \mathcal{N}_{\lambda_{\ast}}^{0}$. Since 
	$$d(\tilde{w},\mathcal{N}_{\lambda_{\ast}}^{0})=\lim_{n\to \infty} d( \tilde{w}_n,\mathcal{N}_{\lambda_{\ast}}^{0})\geq d^{-}>0,$$ we have that 
	$\tilde{w}\not\in \mathcal{N}_{\lambda_{\ast}}^{0}$, that is,   $\tilde{w}\in \mathcal{N}_{\lambda_{\ast}}^{-}$. 
	
	To conclude the proof of the claim, we just need to show that $\tilde{w}\in S_{\lambda_{\ast}}^{-}$. First note that similar arguments as done in the proof of Proposition \ref{PPP1}-$a)$ proves that  $t^{-}_{\lambda_n}(\tilde{w}_{n})\rightarrow t\in (0,\infty)$. From the strong convergence $\tilde{w}_{n}\to \tilde{w}$ in $X$, we get that $\phi^{'}_{\lambda_{\ast},\tilde{w}}(t)=0$ and  $\phi^{''}_{\lambda_{\ast},\tilde{w}}(t)\leq 0$, which lead us to conclude that  $t=1$ since $\tilde{w}\in \mathcal{N}_{\lambda_{\ast}}^{-}$ and  Proposition \ref{prop21}. From Proposition \ref{PPP3} and the strong convergence again, we obtain
	$$\Phi_{\lambda_{\ast}}(\tilde{w})=\lim_{\lambda_{n}\downarrow \lambda_{\ast}}\Phi_{\lambda_{n} }(t_{\lambda
		_{n}}(\tilde{w_{n}})\tilde{w_{n}})=\hat{J}^{-}_{\lambda_{\ast}},$$
	which means that $\tilde{w}\in S_{\lambda_{\ast}}^{-}$.
	Therefore, from Corollary \ref{139} we conclude that
	$\displaystyle \limsup_{\lambda \downarrow \lambda_{\ast}}||\tilde{w}_{\lambda}||\leq ||\tilde{w}||\leq C_{\lambda_{\ast}}$. 
	Since $C>C_{\lambda_{{\ast}}}$, the claim is true. This ends the proof of the claim.
	
	To complete the proof that $w_{\lambda}:=t^{-}_{\lambda}(\tilde{w}_{\lambda})\tilde{w}_{\lambda}\in \mathcal{N}_{\lambda}^{-}$ is a solution to Problem $(P_\lambda)$, let us perturb  $\tilde{w}_{\lambda}$ by appropriate elements of $X_+$ and perform projections of it over $\mathcal{N}_{\lambda_{\ast},d^{-},C}^{-}$ and $\mathcal{N}_{\lambda}^{-}$ Let $\psi \in X_{+}$ and  $\lambda \in (\lambda_{\ast},\lambda_{\ast}+\epsilon_{1})$. Since $\tilde{w}_{\lambda}\in \mathcal{N}_{\lambda_{\ast}}^{-}$, we are able to  apply the implicit function Theorem, as done in Lemma \ref{LL3} $b)$, to prove that $t^{-}_{\lambda_{\ast}}(\tilde{w}_{\lambda}+\theta \psi)$ (see Proposition \ref{prop21}) is well defined, is continuous for $\theta>0$ small enough and $t^{-}_{\lambda_{\ast}}(\tilde{w}_{\lambda}+\theta \psi)\longrightarrow 1$ as $\theta\longrightarrow 0$. 
	
	Thus, it follows from  (\ref{143}) and $d(\tilde{w}_{\lambda},\mathcal{N}_{\lambda_{\ast}}^{0})> d^{-}$ (see definition of $\mathcal{N}_{\lambda_{\ast},d^{-},C}^{-}$)  that 
	$$
	||t^{-}_{\lambda_{\ast}}(\tilde{w}_{\lambda}+\theta \psi)(\tilde{w}_{\lambda}+\theta \psi)||<C~\mbox{and} ~d(t^{-}_{\lambda_{\ast}}(\tilde{w}_{\lambda}+\theta \psi)(\tilde{w}_{\lambda}+\theta \psi),\mathcal{N}_{\lambda_{\ast}}^{0})> d^{-},
	$$
	holds for $\theta>0$ small enough, which implies 
	\begin{equation}\label{145}
		t^{-}_{\lambda_{\ast}}(\tilde{w}_{\lambda}+\theta \psi)(\tilde{w}_{\lambda}+\theta \psi)\in \mathcal{N}_{\lambda_{\ast},d^{-},C}^{-}.
	\end{equation}

	Therefore, by (\ref{145}) and Corollary \ref{CCC1}, we obtain
	$$t_{\lambda}(\theta)t^{-}_{\lambda_{\ast}}(\tilde{w}_{\lambda}+\theta \psi)(\tilde{w}_{\lambda}+\theta \psi)=:t_{\lambda}^{-}(t^{-}_{\lambda_{\ast}}(\tilde{w}_{\lambda}+\theta \psi)(\tilde{w}_{\lambda}+\theta \psi))t^{-}_{\lambda_{\ast}}(\tilde{w}_{\lambda}+\theta \psi)(\tilde{w}_{\lambda}+\theta \psi)\in \mathcal{N}_{\lambda}^{-}$$
	 By  applying Proposition \ref{PPP4}, we have
	$$\Phi_{\lambda}(t_{\lambda}(\theta)t^{-}_{\lambda_{\ast}}(\tilde{w}_{\lambda}+\theta \psi)(\tilde{w}_{\lambda}+\theta \psi))={J}^{-}_{\lambda}(t^{-}_{\lambda_{\ast}}(\tilde{w}_{\lambda}+\theta \psi)(\tilde{w}_{\lambda}+\theta \psi))\geq
	\tilde{J}^{-}_{\lambda,d^{-},C}= \Phi_{\lambda}(t_{\lambda}^{-}(\tilde{w}_{\lambda})\tilde{w}_{\lambda}),$$
	which lead us to conclude that
	\begin{equation}
		\label{1401}
		\Phi_{\lambda}(t_{\lambda}(\theta)t^{-}_{\lambda_{\ast}}(\tilde{w}_{\lambda}+\theta \psi)(\tilde{w}_{\lambda}+\theta \psi))\geq \Phi_{\lambda}(t_{\lambda}^{-}(\tilde{w}_{\lambda})t^{-}_{\lambda_{\ast}}(\tilde{w}_{\lambda}+\theta \psi)\tilde{w}_{\lambda}),
	\end{equation}
	holds for all $\theta>0$ small enough, after using  Proposition \ref{prop21}.
	
	Again, due to the fact that $t_{\lambda}^{-}(\tilde{w}_{\lambda})\tilde{w}_{\lambda}\in \mathcal{N}_{\lambda}^{-}$, we are able to apply the implicit function Theorem, as in Lemma \ref{LL3} $b)$ with the same function $F$ at the point
	$$\Big(t_{\lambda}^{-}(\tilde{w}),||\tilde{w}_{\lambda}||^{2},\lambda\int_{\mathbb{R}^{N}}a|\tilde{w}_{\lambda}|^{1-\gamma},\int_{\mathbb{R}^{N}}b|\tilde{w}_{\lambda}|^{p+1}\Big)$$
	to show  that $t_{\lambda}(\theta)\rightarrow t_{\lambda}^{-}(\tilde{w})$ as $\theta \rightarrow 0$. Since (\ref{1401}) can be read as
	\begin{align*}
		(t_{\lambda}(\theta)t^{-}_{\lambda_{\ast}}(\tilde{w}_{\lambda}+\theta \psi))^{2}\frac{\left[||\tilde{w}_{\lambda}+\theta \psi||^{2}-||\tilde{w}_{\lambda}||^{2}\right]}{\theta}-\frac{(t_{\lambda}(\theta)t^{-}_{\lambda_{\ast}}(\tilde{w}_{\lambda}+\theta \psi))^{p+1}}{p+1}\int \frac{b(\tilde{w}_{\lambda}+\theta \psi)^{p+1}-b(\tilde{w}_{\lambda})^{p+1}}{\theta}\\
		\geq \frac{(t_{\lambda}(\theta)t^{-}_{\lambda_{\ast}}(\tilde{w}_{\lambda}+\theta \psi))^{1-\gamma}}{1-\gamma}\lambda\int \frac{a(\tilde{w}_{\lambda}+\theta \psi)^{1-\gamma}-a(\tilde{w}_{\lambda})^{1-\gamma}}{\theta},
	\end{align*}
	we can follow the arguments done in Lemma \ref{LL4}, Fatou's Lemma and $t_{\lambda}(\theta)\rightarrow t_{\lambda}^{-}(\tilde{w})$ as $\theta \rightarrow 0$, to infer that
	\begin{align*}
		(t_{\lambda}^{-}(\tilde{w}_{\lambda}))^{2}(\tilde{w}_{\lambda},\psi)-(t_{\lambda}^{-}(\tilde{w}_{\lambda}))^{p+1}\int b \tilde{w}_{\lambda}^{p}\psi \geq (t_{\lambda}^{-}(\tilde{w}_{\lambda}))^{1-\gamma}\lambda \int a\tilde{w}_{\lambda}^{-\gamma}\psi,
	\end{align*}
	that is,
	$$(w_{\lambda},\psi)-\int b w_{\lambda}^{p}\psi\geq \lambda \int a w_{\lambda}^{-\gamma}\psi.$$
	  to show To conclude that $w_{\lambda}\in \mathcal{N}_{\lambda}^{-}$ is a solution from $(P_{\lambda})$, we do as in Proposition \ref{T1}.
	\medskip
	
	To complete the proof of Proposition \ref{T2}, let us follow the arguments  done just above with minors adjustments. First, by setting $u_{\lambda}=t^{+}_{\lambda}(\tilde{u}_{\lambda})\tilde{u}_{\lambda} \in \mathcal{N}_{\lambda}^{+}$, with $\tilde{u}_{\lambda}\in \mathcal{N}_{\lambda_{\ast},d^{+},c}^{-}$ being the minimizer of ${J}^{+}_{\lambda}$ constrained to $\mathcal{N}_{\lambda_{\ast},d^{+},c}^{+}$ as given in Proposition \ref{PPP5}, and adjusting the proof of the above claim, we also prove the below claim.
	
	\noindent{\bf Claim:} there exists an $\epsilon_2>0 $ such that 
	$$
	||\tilde{u}_{\lambda}||>c,~\forall ~\lambda \in (\lambda_{\ast},\lambda_{\ast}+\epsilon_{2}),
	$$
	where $c<c_{\lambda_{{\ast}}}$ and $c_{\lambda_{{\ast}}}>0$ is given by Corollary \ref{139}.
	
	After this claim,  by  perturbing  $\tilde{u}_{\lambda}$ by appropriate elements of $X_+$, performing projections of it over $\mathcal{N}_{\lambda_{\ast},d^{+},c}^{+}$ and $\mathcal{N}_{\lambda}^{+}$ and following the same strategy, we can prove that $u_{\lambda}\in \mathcal{N}_{\lambda}^{+}$ is a solution from $(P_{\lambda})$.
	
	Finally, the proof of Proposition follows by taking  $\epsilon=\min\left\{\epsilon_{1},\epsilon_{2}\right\}>0$, that is,  for each $\lambda \in (\lambda_{\ast},\lambda_{\ast}+\epsilon)$ the problem $(P_{\lambda})$ admits at least two solutions $u_{\lambda}\in \mathcal{N}_{\lambda}^{+}$ and $w_{\lambda}\in \mathcal{N}_{\lambda}^{-}$. This ends the proof.
\end{proof}
\fim
\vspace{0.4cm}

\section{Proof of Theorems}

In these section, we are going to prove the main Theorems. Let us begin ending the proof of Theorem \ref{TP1}.

\noindent\textbf{Proof of the Theorem \ref{TP1}.}
First, we note that the multiplicity is given by Propositions  \ref{T1}, \ref{T2a} and \ref{T2}. About qualitative statements, we point out that 
$a)$ is a consequence of Proposition \ref{prop21}. The statement $b)$ follows from Lemma \ref{LL1} and Sobolev embeddings. Let us prove $c)$. To prove that $u_{\lambda}$ is a ground state solution for each $0<\lambda\leq \lambda_{{\ast}}$, let us assume that $w$ is another  solution for Problem $(P_{\lambda})$.  Then $w\in \mathcal{N}_{\lambda}^{+} \cup \mathcal{N}_{\lambda}^{-}$ by either Lemma \ref{L1} $a)$ or Corollary \ref{C1}. If $w\in \mathcal{N}_{\lambda}^{+}$, then $\Phi_{\lambda}(u_{\lambda})=\tilde{J}_{\lambda}^{+}\leq \Phi_{\lambda}(w)$ by definition of $\tilde{J}_{\lambda}^{+}$. On the other hand, if $w\in \mathcal{N}_{\lambda}^{-}$, it follows from  Proposition \ref{prop21} and definition of $\tilde{J}_{\lambda}^{+}$ that  $\Phi_{\lambda}(w)>\Phi_{\lambda}(t^{+}_{\lambda}w)\geq \tilde{J}_{\lambda}^{+}=\Phi_{\lambda}(u_{\lambda})$ holds. So, combining both cases, we conclude that  $u_{\lambda}$ is a ground state solution for Problem $(P_\lambda)$. Now, by (\ref{123}) we have that $\Phi_{\lambda}(u_{\lambda})<0$ for all $0<\lambda<\lambda_{{\ast}}+\epsilon$. From (\ref{N2}) and Sobolev embeddings, we have that $\displaystyle \lim_{\lambda \to 0}||u_{\lambda}||=0$. The  statement $d)$ follow from Propositions \ref{CC1}, \ref{PPP3}.

Finally, let us prove  the item $e)$. First, we note that $\hat{\lambda}$ and $\lambda_{{\ast}}$, as defined at (\ref{eq41}) and (\ref{E3}), respectively, are such that $\hat{\lambda}<\lambda_{{\ast}}$ and $\hat{\lambda} = \inf\{\hat{\lambda}(w):w \in X_+~\mbox{and }\int_{\mathbb{R}^{N}} b
|w|^{p+1}>0 \}$, where $(\hat{\lambda}(w),\hat{t}(w))$ is the unique solution of the system 
$\phi_{\lambda,w}(t)=0, \phi^{'}_{\lambda,w}(t)=0$. So, it follows from Proposition \ref{T1} that there exists a $w_{\hat{\lambda}} \in \mathcal{N}_{\hat{\lambda}}^{-}$ solution of Problem $(P_{\hat{\lambda}})$.

Now, by applying  Proposition \ref{prop21}, we obtain that
\begin{equation}
	\label{141a}
	\Phi_{\lambda}(w_{\lambda})=\phi_{\lambda,w_{\lambda}}(1)\geq \phi_{\lambda,w_{\lambda}}(t(w_{\lambda})) = \Phi_{\lambda}(t(w_{\lambda})w_{\lambda})>\Phi_{\hat{\lambda}( w_{\lambda})}(t(w_{\lambda})w_{\lambda})=0,
\end{equation}
holds	for each $0<\lambda <\hat{\lambda}$ given, where $w_{\lambda} \in \mathcal{N}_{\lambda}^{-}$ is the solution of $(P_{\lambda})$ given by Proposition \ref{T1}. 

On the other hand,  by proceeding as done in Lemma \ref{L0}, we are able to prove that there exists a $w \in X_+$ such that $\hat{\lambda}=\hat{\lambda}(w)$ and  $\Phi_{\hat{\lambda}}(w)=\phi_{\hat{\lambda},w}(1)=\phi^{'}_{\hat{\lambda},w}(1)=0$
. Hence, the Proposition \ref{prop21} imply that $t^{-}(w_{\hat{\lambda}})=1$, which lead us  
\begin{equation}\label{149}
	0=\Phi_{\hat{\lambda}}(w)\geq \Phi_{\hat{\lambda}}(w_{\hat{\lambda}})=\tilde{J}^{-}_{\hat{\lambda}}.
\end{equation}

As a consequence of (\ref{141a}) and of the fact that   $\tilde{J}^{-}_{\lambda}$ is a decreasing and left-continuous function, we have that $\Phi_{\hat{\lambda}}(w_{\hat{\lambda}})=\tilde{J}^{-}_{\hat{\lambda}}\geq 0$. So, this inequality  together with (\ref{149}) lead us to  conclude that $\tilde{J}^{-}_{\hat{\lambda}}=\Phi_{\hat{\lambda}}(w_{\hat{\lambda}})=0$. The rest of the proof  follows from the fact that the function $\tilde{J}^{-}_{\lambda}=\Phi_{\lambda}(w_{\lambda})$ is decreasing for $0<\lambda<\lambda_{{\ast}}+\epsilon$, as showed in Propositions \ref{CC1}, \ref{PPP3}. 
\fim

Below, we are going to prove Theorem \ref{TP2}.

\noindent\textbf{Proof of the Theorem \ref{TP2}.} Take a smooth bounded domain $\Omega \subset \mathbb{R}^N$ and consider the eigenvalue problem
\begin{equation}\label{rs}\tag{$A_{\Omega}$}
	\left\{
	\begin{aligned}
		-&\Delta u +V(x)u = \lambda m(x)u
		~\mbox{in}~ \Omega \\
		& u>0~\mbox{in }\Omega,~~u\in H^{1}_{0}(\Omega), \nonumber
	\end{aligned}
	\right.
\end{equation}
where $m(x)=\min\left\{a(x),b(x)\right\}$. So, as a consequence of Theorem 3 in Brezis-Nirenberg \cite{BN}, we have.
\begin{lemma}\label{APB} The  first eigenvalue $\lambda_{1}$ of the problem \eqref{rs} is positive. Moreover, its associated eigenfunction $e_1$ is positive, $e_{1}\in C^{1}(\overline{\Omega})\cap H^{2}(\Omega)$ and $\partial e_{1}/\partial \nu \leqslant 0$ on $\partial \Omega$, where $\nu \in \mathbb{R}^N$ is the unit exterior normal to $\partial \Omega$.
\end{lemma}

\begin{proof}[Proof of Theorem \ref{TP2}] Let us define  $g:(0,\infty)\rightarrow \mathbb{R}$ by
	$g(t)=\lambda t^{-\gamma-1}+t^{p-1}$ and note that 
	$$t_{\lambda}=\left(\lambda \left(\frac{\gamma+1}{p-1}\right)\right)^{\frac{1}{p+\gamma}},~\lambda>0,$$
	is the its  unique global minimum whose minimum value is given by
	$$\tilde{g}(\lambda):=g(t_{\lambda})=\lambda^{\frac{p-1}{p+\gamma}}\left(\frac{\gamma+1}{p-1}\right)^{\frac{-\gamma-1}{p+\gamma}}\left(\frac{p+\gamma}{p-1}\right),$$
	which provides the existence of  a $\lambda^{\ast}>0$ such that $\tilde{g}(\lambda^{\ast})=\lambda_{1}$, that is,
	$$\lambda^{\ast}=\lambda_{1}^{\frac{p+\gamma}{p-1}} \left(\frac{\gamma+1}{p-1}\right)^{\frac{\gamma+1}{p-1}}\left(\frac{p-1}{p+\gamma}\right)^{\frac{p+\gamma}{p-1}}.$$
	
	Assume that   $u_{\lambda} \in X_+$ is a solution for Problem $(P_{\lambda})$. By Brezis-Nirenberg Theorem (see \cite{BN} Theorem 3 again), we have that $u_{\lambda}^{-\gamma}\in L^{\infty}(K)$ for every $K\subset \subset \Omega$ which implies by  Theorem $12.2.2$ (see J. Jost   \cite{J}) that $u_{\lambda}\in H^{2,\frac{2^{\ast}}{p}}(K)$ and
	$$-\Delta u_{\lambda} =\lambda a(x)u_{\lambda}^{-\gamma}+b(x)u_{\lambda}^{p}-V(x)u_{\lambda}~\mbox{a. e.}~ \mbox{in}~\Omega,$$ 
	and after a classical bootstrap argument, we obtain   that $u_{\lambda}\in H^{2}(\Omega)\cap C(\overline{\Omega})$. Now we apply  Lemma 3.5 of Figueiredo-Gossez-Ubilla \cite{FGU} to conclude that 
	\begin{equation}\label{B1}
		\int_{\Omega}\nabla u_{\lambda}\nabla e_{1}+V(x)e_{1}u_{\lambda}\leq \lambda_{1} \int_{\Omega}m(x)e_{1}u_{\lambda}.
	\end{equation}
	So, it follows from the definition of $\lambda^\ast$, (\ref{B1})  and the fact that $u_{\lambda}$ is a solution fo Problem $(P_\lambda)$, that
	\begin{align*}
		\int_{\Omega}m(x)(\lambda^{\ast} u_{\lambda}^{-\gamma}+u_{\lambda}^{p})e_{1}&\geq \lambda_{1}\int_{\Omega}m(x)u_{\lambda} e_{1}\geq \int_{\Omega}\nabla e_{1}\nabla u_{\lambda}+V(x)e_{1}u_{\lambda}\\
		&=\lambda\int_{\Omega}a(x)u_{\lambda}^{-\gamma}e_{1}+\int_{\Omega}b(x)u_{\lambda}^{p}e_{1}.
	\end{align*}
	Since $a(x),b(x)\geq m(x)$ in $\Omega$, the last inequality lead us to 
	\begin{align*}
		\lambda^{\ast} \int_{\Omega}a(x)u_{\lambda}^{-\gamma}e_{1}+\int_{\Omega}b(x)u_{\lambda}^{p}e_{1}\geq \int_{\Omega}m(x)(\lambda^{\ast} u_{\lambda}^{-\gamma}+u_{\lambda}^{p})e_{1}\geq \lambda\int_{\Omega}a(x)u_{\lambda}^{-\gamma}e_{1}+\int_{\Omega}b(x)u_{\lambda}^{p}e_{1},
	\end{align*}
	which implies that $\lambda^{\ast}\geq \lambda$. This ends the proof.
\end{proof}

\fim


\end{document}